\documentclass[12pt,reqno]{amsart}
\usepackage{amssymb}
\usepackage{amsmath}
\usepackage{enumitem,colonequals}

\usepackage{color}
\usepackage{mathrsfs}
\usepackage[all]{xy}
\usepackage[margin=1.28in]{geometry}

\RequirePackage{mathrsfs} 

\newtheorem{theorem}{Theorem}[section]
\newtheorem{lemma}[theorem]{Lemma}
\newtheorem{proposition}[theorem]{Proposition}
\newtheorem{corollary}[theorem]{Corollary}
\newtheorem{conjecture}[theorem]{Conjecture}

\theoremstyle{definition}
\newtheorem*{ack}{Acknowledgements}
\newtheorem{remark}[theorem]{Remark}
\newtheorem{assumption}[theorem]{Assumption}
\newtheorem{example}[theorem]{Example}
\newtheorem{definition}[theorem]{Definition}

\numberwithin{equation}{section} \numberwithin{figure}{section}

\DeclareMathOperator{\ad}{an}
\DeclareMathOperator{\an}{an}
\DeclareMathOperator{\alg}{alg}

 \DeclareMathOperator{\Spec}{Spec}

\DeclareMathOperator{\Spa}{Spa}
\DeclareMathOperator{\Hom}{Hom}

\newcommand{\ra}{\rightarrow}

\newcommand\ZZ{\mathbb{Z}}

\newcommand\QQ{\mathbb{Q}}
\newcommand\RR{\mathbb{R}}
\newcommand\CC{\mathbb{C}}

\newcommand\OO{\mathcal{O}}
\newcommand\mfX{\mathfrak{X}}
\newcommand\mfU{\mathfrak{U}}

\renewcommand{\leq}{\leqslant}

\renewcommand{\geq}{\geqslant}

\title[Non-archimedean  hyperbolicity]{Non-archimedean hyperbolicity and applications}

\author{Ariyan Javanpeykar}
\address{Ariyan Javanpeykar \\
	Institut f\"{u}r Mathematik\\
	Johannes Gutenberg-Universit\"{a}t Mainz\\
	Staudingerweg 9, 55099 Mainz\\
	Germany.}
\email{peykar@uni-mainz.de}

\author{Alberto Vezzani} 
\address{Alberto Vezzani \\
	Dipartimento di Matematica ``F. Enriques"\\
	Universit\`a degli Studi di Milano\\
	Via Cesare Saldini 50\\
	20133 Milan \\
	Italy.}
\email{alberto.vezzani@unimi.it}

\thanks{The first named author gratefully acknowledges support from  SFB/Transregio 45, and the second named author from the {\it Agence Nationale de la Recherche}, 
	projects ANR-14-CE25-0002 and  ANR-18-CE40-0017.} 
 


\begin{document}

\begin{abstract}  
 Inspired by the work of Cherry, we introduce and study a  new notion of  Brody hyperbolicity for  rigid analytic varieties over a non-archimedean field $K$ of characteristic zero. 
 We use this  notion of hyperbolicity  to  show the following algebraic statement: if    a projective variety     admits a non-constant morphism from an abelian variety, then so does  any specialization of it.  
  As an application of this result, we    show that  the   moduli space of abelian varieties is $K$-analytically Brody hyperbolic  in equal characteristic $0$. These two  results are  predicted by  the Green--Griffiths--Lang  conjecture on hyperbolic varieties and its natural analogues for non-archimedean hyperbolicity.   Finally, we use Scholze's uniformization theorem to prove that the aforementioned moduli space   satisfies a non-archimedean analogue of the ``Theorem of the Fixed Part''    in mixed characteristic.      
\end{abstract}

\maketitle

\thispagestyle{empty}

\section{Introduction}  
 Conjectures of Green--Griffiths and Lang predict a precise interplay between different notions of hyperbolicity \cite{GrGr, Lang2}.  
In this paper, inspired by work of Cherry \cite{CherryKoba}, we introduce and study a new \emph{non-archimedean} notion of hyperbolicity, and study an analogue of the Green--Griffiths--Lang conjecture in this context.

 \subsection{Non-archimedean   Green--Griffiths--Lang's conjecture} Let $X$ be a   variety over $\CC$, and let $X^{\an}$ be the associated complex analytic space. Recall that $X$ is Brody hyperbolic if $X^{\an}$ has no entire curves, i.e., every holomorphic map $\mathbb{C}\to X^{\an}$ is constant. 
A conjecture of Green--Griffiths--Lang says that, if every algebraic map from an abelian variety to a projective variety $X$ over $\CC$ is constant, then $X^{\an}$ is Brody hyperbolic; see \cite{GrGr, Lang2}. 
We now formulate a non-archimedean analogue of this conjecture.

If $K$ is a complete non-archimedean valued field and $X$ is a finite type scheme over $K$, we let $X^{\an}$  be the associated rigid analytic variety over $K$. We say that a variety over $K$ is \emph{$K$-analytically Brody hyperbolic} if, for every finite type connected group scheme $G$ over $K$, every morphism $G^{\an}\to X^{\an}$ is constant; see    Section \ref{section:definitions_hyperbolicity} for more definitions.  The analogue of the Green--Griffiths--Lang conjecture in this context reads as follows.

 \begin{conjecture}[Non-archimedean~Green--Griffiths--Lang]\label{conj:lang}  Let $K$ be an algebraically closed complete non-archimedean valued field of characteristic zero, and let $X$ be a proper scheme over $K$. Suppose that, for every abelian variety $A$ over $K$, every morphism $A\to X$ is constant. Then $X$ is $K$-analytically Brody hyperbolic. 
\end{conjecture}

 Although the  notion   of $K$-analytic Brody hyperbolicity introduced above   has not appeared before in the literature, we were first led to investigate  this notion by the work of Cherry; see      \cite{AnCherryWang, Cherry, CherryKoba, CherryRu, LevinWang, LinWang}.  
 
There is no non-constant morphism $\mathbb A^{1,\ad}_{K}\to \mathbb G_{m,K}^{\ad}$ (contrary to the complex analytic setting); see \cite[Proposition~I.4]{Cherry} for instance. Therefore, as $\mathbb{G}_{m,K}$ is clearly not hyperbolic, we see that, to prove that an algebraic variety $X$ over $K$ is  $K$-analytically Brody hyperbolicity, it does not suffice to show that all morphisms $\mathbb{A}^{1,\ad}_{K}\to X^{\ad}$  are constant. 
 
In fact, to ``show'' the $K$-analytic Brody hyperbolicity of a variety $X$ over $K$,  one has to verify that, for every connected algebraic group $G$ over $K$, every analytic map $G^{\an}\to X^{\an}$ is constant. We show that, a variety $X$ over $K$ is $K$-analytically Brody hyperbolic if and only if, for every abelian variety $A$ over $K$ with good reduction every morphism of $K$-schemes $A\to X$ is constant \emph{and} every $K$-analytic morphism $\mathbb{G}_{m,K}^{\an}\to X^{\an}$ is constant.
 
 We stress that it is not clear what the ``right'' notion of non-archimedean hyperbolicity is (or should be). In this paper, we investigate a non-archimedean analogue of ``groupless varieties'' (as defined   in \cite{JKam, JXie} and   Definition \ref{def:groupless}), and show that this notion satisfies some of the expected properties. However, it is certainly worth pursuing non-archimedean analogues of ``Borel hyperbolic'' complex algebraic varieties \cite{JKuch}, as recently done by Ruiran Sun \cite{Sun}.   Moreover, another perspective on complex-analytic hyperbolicity is provided by Kobayashi's pseudometric. Cherry proposed a non-archimedean analogue of this pseudometric, but notices quickly that it does not have the right properties. We refer the reader to Section \ref{section:distance} for a discussion of Cherry's non-archimedean analogue of Kobayashi's pseudodistance.
 
 Lang   conjectured that, for $k$ an algebraically closed field of characteristic zero and $X$ a  projective variety $X$ over $k$, we have that every morphism from an abelian variety $A$ to $X$ is constant   if and only if it is ``arithmetically hyperbolic'' \cite[Definition~4.1]{JLalg}, i.e., for every $\ZZ$-finitely generated subring $A\subset \CC$, and every finite type separated scheme $\mathcal{X}$ over $A$ with $\mathcal{X}_\CC\cong X$, the set $\mathcal{X}(A)$ is finite; see   \cite[\S0.3]{Abr}, \cite{JBook}, \cite[Conjecture~1.1]{JAut}, \cite{Lang2}  and \cite[Conjecture~XV.4.3]{CornellSilverman}.
 In conclusion, the   non-archimedean version of the Green--Griffiths--Lang conjecture (Conjecture \ref{conj:lang})   predicts that a projective  variety $X$  over an algebraically closed complete non-archimedean valued field $K$ of characteristic zero is arithmetically hyperbolic over $K$ if and only if it is $K$-analytically Brody hyperbolic. Other ``arithmetic'' speculations related to $K$-analytic Brody hyperbolicity are   made by An--Levin--Wang  \cite{AnLevinWang}.
 
 \subsection{Evidence for   non-archimedean Green--Griffiths--Lang}
 Our first result verifies Conjecture \ref{conj:lang} for constant varieties over $K$.
 
 \begin{theorem}[Non-archimedean Green--Griffiths--Lang for constant varieties]\label{thm:lang_for_constant} 
 Let $X$ be a proper  scheme over an algebraically closed field $k$. Fix a complete  algebraically closed non-archimedean valued field $K$ with ring of integers $\OO_K$ and fix a section $k\hookrightarrow\OO_K$ of the quotient map. 
 Suppose that, for every abelian variety $A$ over $K$, every morphism $A\to X_K$ is constant. Then $X_K$ is $K$-analytically Brody hyperbolic over $K$.
 \end{theorem}
 
 To prove  Theorem \ref{thm:lang_for_constant}, we establish a more general result relating $K$-analytic Brody hyperbolicity of the generic fiber of a proper scheme $X$ over $\OO_K$ to the ``hyperbolicity'' of its special fiber. The precise statement reads as follows.

\begin{theorem}[Inheriting   hyperbolicity from the special fiber]\label{thm:gen_fiber_is_hyp} Let $K$ be complete algebraically closed non-archimedean valued field $K$.
Let $X$ be a proper finitely presented scheme over $\OO_K$. Suppose that, for every abelian variety over $k$, every morphism $A\to X_k$ is constant. Then $X_K$ is $K$-analytically Brody hyperbolic. 
\end{theorem}

Theorem \ref{thm:gen_fiber_is_hyp} is proven using a uniformization theorem of Bosch--L\"utkebohmert for abelian varieties over $K$.

The above two results (Theorems \ref{thm:lang_for_constant} and \ref{thm:gen_fiber_is_hyp})   are in accordance with the non-archimedean Green--Griffiths--Lang conjecture (Conjecture \ref{conj:lang}).  We note that   Conjecture \ref{conj:lang} was previously shown to hold  for closed subvarieties of abelian varieties. More precisely, for closed subvarieties of abelian varieties,  the work of  Cherry \cite{Cherry}, Faltings \cite{FaltingsLang1, FaltingsLang2}, Kawamata \cite{Kawamata}, and Ueno \cite[Theorem~3.10]{Ueno} can be combined into the following result.

 \begin{theorem}[Cherry, Faltings, Kawamata, Ueno]\label{thm:bloch} Let $K$ be an algebraically closed field of characteristic zero.
 Let $X$ be a closed subvariety of an abelian variety over $K$. Then the following are equivalent.
 \begin{enumerate}
 \item For every abelian variety $A$ over $K$, every morphism $A\to X$ is constant, i.e., $X$ is groupless (Definition \ref{def:groupless}). (Equivalently, $X$ does not contain the translate of a positive-dimensional abelian subvariety of $A$.)
 \item Every closed integral subvariety of $X$ is of general type.
 \item The projective variety $X$ is arithmetically hyperbolic \cite[Definition~4.1]{JLalg}.
 \item  If $K=\CC$: The projective variety is $X$ is Brody hyperbolic.
  \item If $K$ is non-archimedean: The projective variety $X$ is $K$-analytically Brody hyperbolic.
 \end{enumerate}
 \end{theorem}

Cherry's theorem (Theorem \ref{thm:bloch}) not only shows that the non-archimedean Green--Griffiths--Lang conjecture holds for closed subvarieties of abelian varieties, but it also suggests that our notion of non-archimedean hyperbolicity is not far from the ``right'' one. 

We   use Cherry's theorem (Theorem \ref{thm:bloch}) to verify that the non-archimedean Green--Griffiths--Lang conjecture  (Conjecture \ref{conj:lang}) holds  for all projective curves (Proposition \ref{prop:lv_for_curves}) and their symmetric powers (Proposition \ref{prop:lang_for_syms}). 

\subsection{An algebraic application of Theorem \ref{thm:gen_fiber_is_hyp}} 
If $X$ is a projective variety over an algebraically closed field $k$, then we say that $X$ is \emph{groupless} if, for every abelian variety $A$ over $k$, every morphism $A\to X$ is constant (see Definition \ref{def:groupless}).  To motivate the following results, note that the classic Green--Griffiths--Lang conjecture predicts that a projective variety over $\CC$ is groupless if and only if it is of general type and Brody hyperbolic.  

Now, the (classical) Green--Griffiths--Lang conjecture predicts that every specialization of a non-groupless variety is non-groupless; we explain this in detail in Remark \ref{remark:gt}. Our next result verifies this prediction.  

\begin{theorem}[Grouplessness generizes]\label{thm:grouplessness_generizes}
  Let $S$ be an integral normal variety over an algebraically closed field $k$ of characteristic zero with function field $K=K(S)$. Let $K\to K^a$ be an algebraic closure of $K$.
	Let $X\to S$ be a proper morphism of  schemes such that there is an $s$ in $S(k)$ with $X_s$ groupless over $k$.   Then the geometric generic fiber $X_{K^a}$ of $X\to S$ is groupless over $K^a$.  
\end{theorem}

It seems worth stressing that our proof of Theorem \ref{thm:grouplessness_generizes} uses our ``algebraic'' criterion for non-archimedean hyperbolicity (Theorem \ref{thm:gen_fiber_is_hyp}).  In particular, our proof of Theorem \ref{thm:grouplessness_generizes} illustrates that our notion of non-archimedean hyperbolicity can be used to verify algebraic predictions made by the (classical) Green--Griffiths--Lang conjecture.

To motivate our next result, we note that the fact that  a groupless variety should specialize to  a groupless variety   is also a consequence of the Green--Griffiths--Lang conjecture; we explain this also in detail in   Remark \ref{remark:gt}.   The precise statement reads as follows.

 \begin{theorem}[Grouplessness specializes]\label{thm:groupless_spec} Let $k$ be an uncountable algebraically closed field. Let $S$ be an integral variety over $k$ with function field $K = K(S)$. Let $K\to K^a$ be an algebraic closure of $K$.
 	Let $X\to S$ be a  proper morphism of schemes. If $X_{K^a}$ is   groupless over $K^a$, then there is an $s$ in $S(k)$ such that $X_s$ is groupless.
 \end{theorem}

To prove Theorem \ref{thm:groupless_spec} we use (only) algebraic arguments, and properties of the moduli space of principally polarized abelian varieties.

We stress that Theorems \ref{thm:grouplessness_generizes} and \ref{thm:groupless_spec} verify predictions made by the Green--Griffiths--Lang conjecture, and thereby provide (new) evidence for this conjecture.

\subsection{The moduli space of abelian varieties}  
If $g$ is a positive integer, we let $\mathcal A_g$ be the stack of $g$-dimensional principally polarized abelian schemes over $\mathbb Z$.  Let $N\geq 3$ be an integer coprime to $p$ and let $\mathcal{X} :=\mathcal{A}_{g}^{[N]}$ be the   moduli space of $g$-dimensional principally polarized abelian schemes with full level $N$-structure over $\ZZ[1/N]$. Note that $\mathcal{X}$ is a smooth quasi-projective scheme over $\mathbb{Z}[1/N]$; see \cite{MoretBailly}.

 The Green--Griffiths--Lang conjecture has an analogue for  quasi-projective, not necessarily projective, varieties to which we will refer to as the Lang--Vojta conjecture.
The original versions of Lang--Vojta's conjecture appeared in   \cite{Lang2} and \cite[Conj.~4.3]{Vojta3}. A general conjecture over finitely generated subrings of $\CC$ is stated in \cite[\S0.3]{Abr}  (see also \cite[Conj.~6.1]{JL}). Part of Lang--Vojta's conjecture predicts that,   if $X$ is a quasi-projective integral scheme over $\CC$ whose integral subvarieties are of log-general type, then $X$ is arithmetically hyperbolic over $\CC$ \cite[Definition~4.1]{JLalg} and Brody hyperbolic.  
 
It follows from the work of   Zuo that every integral subvariety of the moduli space $\mathcal{X}_\CC$ is of log-general type \cite{ZuoNeg} (see also \cite{Brunebarbe} and \cite[Lemma~6.3]{JL}). The aforementioned predictions made by the Lang--Vojta conjecture are that $\mathcal{X}_\CC$ is arithmetically hyperbolic over $\CC$ and Brody hyperbolic.  The arithmetic hyperbolicity of $\mathcal{X}_\CC$ was proven by Faltings; see \cite{Faltings2, FaltingsComplements}. As it plays a crucial role in our work below, recall  that $\mathcal{X}_{\CC}$ is (complex analytically) Brody hyperbolic, as  any holomorphic map  $\mathbb{A}^{1\ad}_{\CC}\to \mathcal{X}_{\CC}^{\textrm{an}}$   lifts   to the universal cover $\mathbb H_g$ of $\mathcal {X}_{\mathbb C}^{\textrm{an}}$,  the universal cover  $\mathbb H_g$ of $\mathcal X_{\mathbb C}^{\textrm{an}}$ is  a bounded domain, and bounded domains are Brody hyperbolic by Liouville's theorem.  In a diagram, the proof can be summarized as follows:
 
 \[
 \xymatrix{  & & \textrm{bounded domain} \ar[d]^{\textrm{topological covering}} \\ \mathbb{A}^{1\ad}_{\CC} \ar@{-->}[urr]^{\exists} \ar[rr] & & \mathcal X_{\CC}^{\ad}  }
 \] This proof also shows that maps into $\mathcal{X}_\CC$ with trivial monodromy are constant. The latter statement is an instance of the so-called ``Theorem of the Fixed Part''; see \cite[Theorem~3.1]{VoisinII} for the more general version due to Griffiths.

\begin{theorem}[Theorem of the Fixed Part]\label{thm:fixed_part_intro} Let $S$ be a finite type separated connected scheme over $\CC$. A  morphism  $S^{\ad}\ra \mathcal X_{\CC}^{\an}$ is constant if and only if the image of the induced morphism on fundamental groups 
\[
\pi_1(S^{\an}) \to \pi_1(\mathcal{X}_{\CC}^{\an})
\] is finite.
\end{theorem}
 
Motivated by the Lang--Vojta conjecture and the aforementioned properties of the moduli space of abelian varieties, it seems reasonable to suspect that, if $K$ is 
 a complete algebraically closed non-archimedean valued field of characteristic zero,  then the moduli space $\mathcal{X}_K$  is $K$-analytically Brody hyperbolic.
Our first result in the direction of this reasonable expectation reads as follows.

\begin{theorem}  
	If the residue field of $K$ has characteristic $0$, then the moduli space $\mathcal{X}_K$  is $K$-analytically Brody hyperbolic.
\end{theorem}

In the mixed characteristic case, we do  not     show that the moduli space $\mathcal{X}_K$ is $K$-analytically Brody hyperbolic. However, we do obtain an analogue of the Theorem of the Fixed Part in this case. 
Indeed, using results of Scholze on perfectoid spaces, an argument   similar to  the above ``complex-analytic'' argument leads to the following result.

\begin{theorem}[Non-archimedean Theorem of the Fixed Part]   \label{thm:moduli_of_abelian_varieties_1} Let  $K$ be a complete algebraically closed non-archimedean valued field of characteristic zero and residue characteristic $p>0$. Let $S$ be a connected rigid analytic variety over $K$. A   morphism $S \ra \mathcal X_{K}^{\ad}$ is constant if and only if the image of the induced morphism on (algebraic) \'etale pro-$p$-fundamental groups 
\[
\pi_1^{\alg}(S) \ra \pi_1^{\alg}(\mathcal{X}_{K}^{\an}) \ra \mathrm{GSp}_{2g}(\ZZ_p)
\]
is finite.
\end{theorem}

  Our   proof of Theorem \ref{thm:moduli_of_abelian_varieties_1}   is conceptually quite close to the line of reasoning sketched above. 
  However, 
instead of lifting a map $S\to \mathcal{X}_{K}^{\ad}$ to the topological universal cover of $\mathcal{X}_{K}^{\ad}$, we lift to a suitable perfectoid pro-finite \'etale covering \cite[\S3]{ScholzeTorsion}.  

In a diagram similar to the above ``complex-analytic'' diagram, our line of reasoning used to prove Theorem \ref{thm:moduli_of_abelian_varieties_1} can be summarized as follows:

 \[
 \xymatrix{       & & \textrm{Scholze's perfectoid space} \ar[d]^{\textrm{pro-finite \'etale}} \\   S \ar[rr] \ar@{-->}[urr]^{\exists}& &  \mathcal{X}_{K}^{\ad}  }
 \]

 For the reader's convenience, we stress that the crucial property of perfectoid spaces used in the proof of Theorem \ref{thm:moduli_of_abelian_varieties_1} is the following folklore result; we refer to Proposition \ref{prop:rigid_to_perf} for a more precise statement.
 
 \begin{proposition}\label{thm:perf_is_hyp}  Let  $K$ be a complete algebraically closed non-archimedean valued field of characteristic zero and residue characteristic $p>0$. 
 A perfectoid space over $K$ is $K$-analytically Brody hyperbolic.
 \end{proposition}

	\begin{ack} We thank Jackson Morrow for very helpful discussions. We thank Peter Scholze for a very helpful and inspiring discussion in Alpbach.   We are grateful to  Giuseppe Ancona, Johannes Ansch\"utz,  Olivier Benoist, Yohan Brunebarbe,  Ana Caraiani,  C\'edric Pepin,  Arno Kret, Robert Kucharczyk, Jaclyn Lang, Marco Maculan, Lucia Mocz, Frans Oort, Will Sawin, Beno\^it Stroh, Yunqing Tang, Jacques Tilouine,  Robert Wilms,  and Kang Zuo for helpful discussions. 
	\end{ack}

\section{Non-archimedean Brody hyperbolicity}\label{section:definitions_hyperbolicity} 
Throughout this paper, we make the following assumption. 
\begin{assumption}  
We let $K$ be an algebraically closed  field of characteristic $0$ which is complete with respect to a non-archimedean, non-trivial multiplicative norm  ${||\cdot||}\colon K\ra\RR$. We let $\OO_K$ be its valuation ring and $k$ be its residue field.
\end{assumption}

\begin{example}
Finite field extensions of $\mathbb{Q}_p$ are complete non-archimedean valued fields of characteristic zero and residue characteristic $p>0$. The completion $\mathbb{C}_p$ of $\overline{\mathbb{Q}_p}$ is also a complete non-archimedean valued of field of characteristic zero and residue characteristic $p>0$. Finally, for any field $k$ of characteristic $0$, the completion $K$ of $\overline{k(\!(t)\!)}$ endowed with the $t$-adic valuation is a complete non-archimedean valued field with residue field $k$.
\end{example}

We will use the language of adic spaces \cite{huber2, huber} in order to deal with rigid analytic varieties. In particular, we will denote by $|X|$ the underlying topological space of a rigid analytic variety as defined by Huber \cite{huber}. 
 All schemes and adic spaces (see \cite{huber2} for the general theory) will be considered as spaces over $K$ endowed with its valuation of rank $1$.  For a reduced Tate algebra $R$ which is a quotient of $K\langle T_1,\ldots, T_N\rangle$ (the completion of the polynomial ring)  we write $\Spa R$ for the affinoid space $\Spa(R,R^\circ)$. We let $\mathbb{B}^N$ be the poly-disc $\Spa K\langle T_1,\ldots,T_N\rangle$.
 
  If $X$ is a locally finite type  scheme over $K$, we let $X^{\ad}$ be the associated rigid analytic variety (see for example \cite[Proposition 0.3.3]{Berth} or \cite{BGR}). To simplify the notation, we will sometimes let $\mathbb G_m^{\ad}$ be the rigid analytic variety associated to the finite type $K$-scheme $\mathbb G_{m,K}$. 
  
  The following definition introduces the main objects of study of our paper. We stress that our definition is modeled on the observation that a complex algebraic variety $X$ over $\mathbb{C}$ is Brody hyperbolic (i.e., has no entire curves $\mathbb{C}\to X^{\an}$) if and only if for every connected finite type group scheme $G$ over $\CC$, every morphism of varieties $G^{\an}\to X^{\an}$ is constant. We refer the reader interested in complex algebraic Brody hyperbolic varieties to Kobayashi's book \cite{Kobayashi}.

\begin{definition}\label{def}
An adic space $X$ over $K$ is \emph{$K$-analytically Brody hyperbolic} if, for every connected finite type group scheme $G$ over $K$, every morphism   of adic spaces  $G^{\ad}\ra X$  is constant, i.e. factors over a $K$-rational point. 
	A locally finite type $K$-scheme   $X$ is \emph{$K$-analytically Brody hyperbolic} if 
	$X^{\ad}$ is $K$-analytically Brody hyperbolic. 
\end{definition}

\begin{remark}
Since any finite type group scheme $G$ over $K$ is reduced (as $K$ is of characteristic zero) and has a $K$-rational point (by definition),  a map $G^{\ad}\ra X$ is constant if and only if it is constant as a map of topological spaces. 
\end{remark}

\begin{remark}
The closed disc $\mathbb{B}^1$ and the open disc $\mathring{\mathbb{B}}^1$  are subgroups   of $(\mathbb{A}^{1,\ad},+)$, and any  rigid analytic variety admits non-trivial maps from $\mathbb{B}^1$. A definition of hyperbolicity using  rigid analytic groups as ``test'' objects would therefore be of no interest.
\end{remark}

   \begin{remark}\label{remark:qf} In \cite[Corollary~3.11.2]{Kobayashi}, it is shown that  the total space of a family of  Kobayashi hyperbolic varieties over $\mathbb{C}$ over a Kobayashi hyperbolic variety over $\mathbb{C}$ is Kobayashi hyperbolic. The non-archimedean analogue of this statement reads as follows (and is not hard to prove).  
 	Let $X\to Y$ be a morphism of   rigid analytic varieties over $K$. If $Y$ is $K$-analytically Brody hyperbolic and all $K$-fibers of $X\to Y$ are $K$-analytically Brody hyperbolic, then $X$ is $K$-analytically Brody hyperbolic.
 \end{remark}

  The notion  of $K$-analytic hyperbolicity introduced above (Definition \ref{def}) has not appeared before in the literature. However, the following weaker notion     was studied in the work of Cherry \cite{CherryKoba}; see also \cite{AnCherryWang, Cherry, CherryRu, LevinWang, LinWang}.
 
 \begin{definition}\label{def:cherry_hyp}
 An adic space $X$ over $K$ is \emph{$K$-analytically pure} if all morphisms $\mathbb G_{m,K}^{\ad}\to X$ are constant. A locally finite type scheme $X$ over $K$ is \emph{$K$-analytically pure} if the adic space $X^{\an}$ is $K$-analytically pure.
 \end{definition}

 Note that smooth proper curves of genus one over  $K$ with good reduction are  $K$-analytically pure. In particular, a $K$-analytically pure is not necessarily ``hyperbolic''.  This is why we avoid saying that $K$-analytically pure varieties are ``hyperbolic''.  
Part of the aim of this paper     is to show that a (more) ``correct'' non-archimedean analogue of Brody hyperbolic complex manifolds is provided by $K$-analytic Brody hyperbolicity (Definition \ref{def}).

\subsection{Hyperbolicity of affinoids}
 
It follows from Liouville's theorem on bounded holomorphic functions that a bounded domain in affine space is Brody hyperbolic.  
In the non-archimedean setting, we can state the following analogue of this fact. 

	\begin{proposition}\label{prop:aff_are_brody}
		An affinoid adic space $A$ over $K$ is $K$-analytically Brody hyperbolic. More generally, if $X$ is a   reduced connected locally finite type scheme over $K$, then any morphism $X^{\ad}\to A$ is constant.
	\end{proposition}
	
	\begin{proof}   It suffices to prove the second statement. To do so, note that any map $X^{\ad}\to A$ to an affinoid space $A=\Spa(R,R^+)$ over $K$ is uniquely determined by a morphism from $R^+$ to the ring of global bounded functions on $X^{\ad}$; see \cite[Proposition 2.1(ii)]{huber2}. The claim then follows if we show that this ring is $\mathcal{O}_K$. Therefore, we may and do assume that  $A$ is the closed disc $\mathbb{B}^{1}$.
	
	Now, to show that $X^{\ad}\to A$ is constant, we may and do assume that $X$ is affine.  Moreover, by considering the normalization of $X$, we can also assume that $X$ is normal. Let $X\to \overline{X}$ be an open immersion with $\overline{X}$ connected, normal and projective  over $K$. By the Hebbarkeitssatz (\cite[p.~502]{conrad-conn} or \cite[Theorem~1.6]{Lutk}), any morphism $X^{\ad}\ra \mathbb{B}^{1}$ extends to a morphism  $\overline{X}^{\ad}\ra\mathbb{B}^{1}$.   By GAGA \cite{conradGAGA, kopfGAGA}, the  composite map  $\overline{X}^{\ad}\ra\mathbb{B}^{1}\ra\mathbb{P}^{1,\ad}$ is induced by an algebraic map $\overline{X}\ra\mathbb{P}^1$ whose image is closed and does not contain   $\infty$. Therefore,  since $\overline{X}$ is connected, we conclude that $\overline{X}^{\ad}\to \mathbb P^{1,\ad}$ is   constant, as required. 
	\end{proof}
	
	\begin{remark}\label{remark:liouville}
	The complex analytic analogue of Proposition \ref{prop:aff_are_brody} reads as follows. Let $D$ be a bounded domain in the affine space $\mathbb{C}^N$, and let $X$ be a   reduced  connected locally finite type  scheme over $\CC$. Then, any morphism $X^{\ad}\to D$ is constant. Indeed, since any two points in $X$ lie in the image of a reduced connected curve, we may and do assume that $X$ is a curve. Let $Y$ be a connected component of the normalization $X'$ of $X$, and note that $Y$ is a smooth quasi-projective curve over $\CC$. It suffices to show that $Y^{\ad}\to D$ is constant. Let $Y\subset \overline{Y}$ be the smooth compactification of $Y$. Let $P\in \overline{Y}\setminus Y$, and let $\Delta_P\subset \overline{Y}^{\ad}$ be an open disk around $P$ not containing any other point of $\overline{Y}\setminus Y$. Then, as the morphism $\Delta_P\setminus \{P\}\to D$ is a bounded holomorphic function, it extends to a morphism $\Delta_P\to D$. We see that $Y^{\ad}\to D$ extends a holomorphic morphism $\overline{Y}^{\ad}\to D\subset \CC^n$. As $\overline{Y}$ is projective, every global holomorphic function $\overline{Y}^{\ad}\to \CC$ is constant. We see that $\overline{Y}^{\ad}\to D$ is constant, and conclude that $Y^{\ad}\to D$ is constant.
	\end{remark}

  \subsection{Hyperbolicity of perfectoid spaces}\label{section:lemma}
Recall that in Proposition \ref{prop:aff_are_brody} we showed that affinoid adic spaces do not admit non-constant maps from a connected algebraic variety. Similarly, perfectoid spaces over $K$ (see \cite{ScholzePerfectoid} for definitions) also do not admit non-constant morphisms from an algebraic variety. Even better, they do not admit any non-constant morphisms from a   connected rigid analytic variety (Proposition \ref{prop:rigid_to_perf}).  This property is fairly obvious to anyone accustomed to Scholze's theory of perfectoid spaces;  we include it here due to a lack of a suitable reference in the literature.
     
   \begin{proposition}\label{prop:rigid_to_perf} Assume $K$ has positive residue characteristic.
   	Let $X$ be a reduced connected rigid analytic variety over $K$ and let $Y$ be a perfectoid space over $K$. Then any  morphism of adic spaces
   	$X\to Y$ is constant.  
   \end{proposition}
   \begin{proof}  By a standard induction argument on the dimension of $X$, replacing $X$ by its (dense open) smooth locus if necessary, we may and do assume that $X$ is smooth.  Moreover, 
   	we may and do assume that  
   	$X$ and $Y$ are affinoid and equal to $\Spa(R,R^\circ)$ and $\Spa (P,P^{+})$, respectively.  Let  $\varphi\colon P\ra R$ be the morphism of rings associated to $X\to Y$.   
   	
   	We first  prove that $R$ does not contain any non-constant element having arbitrary $p$-th roots. (That is, we first treat the case   $P=K\langle T^{1/p^{\infty}}\rangle$.)  Fix a $K$-rational point $x$ of $X$. Since $X$ is smooth, it is locally \'etale over a poly-disc $\mathbb{B}^n$ (see \cite[Corollary 1.1.51]{Ayoub}) and there is an open neighborhood  of $x$ which is isomorphic to an open neighborhood of a rational point $y$ in $\mathbb{B}^n$ (see \cite[Remark 1.2.4]{Ayoub}). Thus, since concentric poly-discs form a basis of open neighborhoods of $y$, by the identity theorem (see \cite[Lemma 2.1.4]{conrad-conn}), we may and do assume that $X$ is $\mathbb{B}^n$, so that $R$ equals the Tate algebra $K\langle T_1,\ldots,T_n\rangle$.  
   	Note that, as in the proof of \cite[Lemma~3.4.(i)]{ScholzePerfectoid}, we have an isomorphism of multiplicative monoids
   	\begin{eqnarray*}
   		\varprojlim_{x\mapsto x^p} \mathcal{O}_K\langle T_1,\ldots,T_n\rangle &\cong& \varprojlim_{x\mapsto x^p} \mathcal{O}_K\langle T_1,\ldots,T_n\rangle/p.
   	\end{eqnarray*}
   	Therefore, we have  the following isomorphisms of multiplicative monoids 
   	\begin{eqnarray*}
   		\varprojlim_{x\mapsto x^p} \mathcal{O}_K\langle T_1,\ldots,T_n\rangle &\cong& \varprojlim_{x\mapsto x^p} \mathcal{O}_K\langle T_1,\ldots,T_n\rangle/p
   		 \\ &\cong& \varprojlim_{x\mapsto x^p} (\mathcal{O}_K/p)[ T_1,\ldots,T_n]\cong \varprojlim_{x\mapsto x^p} \mathcal{O}_K/p  \cong \varprojlim_{x\mapsto x^p} \mathcal{O}_K.
   	\end{eqnarray*}
   	We deduce that $R^\circ = \OO_K\langle T_1,\ldots, T_n\rangle$ does not contain any non-constant element with arbitrary $p$-th roots.  By multiplying with a constant, this implies the statement for $R$ as well, as desired.
   	
  	Let $$P^\flat\cong\varprojlim_{x\mapsto x^p} P$$ be the tilt of $P$; see \cite[Proposition 5.17]{ScholzePerfectoid}. Consider the  multiplicative map $\sharp\colon P^\flat\ra P$, $y\mapsto y^\sharp$ given by the projection to the first term. Since $R$ does not contain any non-constant element having all $p$-th roots (as shown above), we deduce that $\varphi(\sharp(P^{\flat}))$ is contained in $K$.  Because the canonical map of rings $ P^{\flat\circ}/p^{\flat}\to P^\circ/p$ is an isomorphism (see \cite[Lemma 6.2]{ScholzePerfectoid}), it follows that, for any element $x$ in $P^\circ$, there exist an element  $y$ in $P^{\flat\circ}$ and an element $x'$ in $P^{\circ}$ such that $x=y^\sharp+px'$. Fix an element $x_0$ in $P^\circ$ and define inductively a sequence $x_i\in P^\circ$ and $y_i\in P^{\flat\circ}$ by 
   	$$
   	x_i=y_i^\sharp+px_{i+1}
   	$$
   	Note that the sequence $(z_n)_{n=0}^\infty$ defined by $z_n=\sum_{i=0}^n p^iy^\sharp_i=x_0-p^{n+1}x_{n+1}$ converges to $x_0$. By what we proved above, we see that $\varphi(z_n)=\sum_{i=0}^n p^i\varphi(y_i^\sharp)\in K$. We deduce that $\varphi(x_0)$ is a limit of elements of $K$, and thus  lies in $K$. This  proves that $\varphi$ factors over $K$ as wanted. 
   \end{proof}

   \begin{remark} 
   	Proposition \ref{prop:rigid_to_perf} can be generalized to a   connected (not necessarily reduced) rigid analytic variety $X$ by saying that its image is a single rational point of $Y$. Indeed, it suffices to take the reduced closed subvariety of $X$.  On  the other hand, we remark that in this (more general) case the morphism may not factor over $\Spa K$. For example, 
   	the map $T^{1/p^{n}}\mapsto 1+\frac{1}{p^n}T$ induces a  morphism from the perfectoid algebra $K\langle T^{1/p^\infty}\rangle$ to the (non-reduced) Tate algebra $K\langle T\rangle/T^2$ which clearly does not factor over $K$.
   \end{remark}
   
      \begin{remark}
   Let $X$ be an adic space over $K$. Let $P\subset X$ be an open adic subspace which is perfectoid.  If $S$ is a connected reduced rigid analytic variety over $K$  and $S\to X$ is a non-constant morphism, then the image of $S\to X$ lies in the complement of $P$. In particular, following the complex-analytic terminology \cite[\S3.2]{Kobayashi},  one could say that $X$ is ``hyperbolic modulo the complement of $P$''. 
   \end{remark}

	\subsection{Descending hyperbolicity along finite \'etale covers}
If $X\to Y$ is a finite \'etale morphism of complex algebraic varieties and  $X$ is Brody hyperbolic (resp. Kobayashi hyperbolic), then $Y$ is Brody hyperbolic (resp. Kobayashi hyperbolic); see for instance \cite[Theorem~3.2.8.(2)]{Kobayashi}. We now show the analogue of the latter statement for $K$-analytically Brody hyperbolic varieties.

\begin{proposition}\label{prop:chevalley_weil} \label{prop:cv_stack}  
	Let $X\to Y$ be a   finite \'etale morphism of adic spaces    over $K$. If $X$ is $K$-analytically Brody hyperbolic, then $Y$ is $K$-analytically Brody hyperbolic.
\end{proposition}
\begin{proof}   
	 Let $G$ be a finite type connected group scheme over $K$,   
	and  let $\varphi:G^{\ad}\to Y^{\ad}$ be  a morphism.   Consider the Cartesian diagram
	\[ 
	\xymatrix{V'' \ar[rr] \ar[d] & & X^{\ad} \ar[d] \\ G^{\ad} \ar[rr]_{\varphi} & & Y^{\ad} } 
	\]  
	
	Let $V'$ be a connected component of $V''$, and note that $V'\to G^{\an}$ is finite \'etale. Let $V\to V'\to G^{\an}$ be its Galois closure,  so that $V\to G^{\an}$ is finite \'etale Galois.
	We claim that there is a finite type group scheme $H$ over $K$ and a central isogeny $H\to G$ such that the associated morphism $H^{\ad}\to G^{\ad}$ is isomorphic to $V\to G^{\ad}$. 
	 To prove this, note that,  by the non-archimedean analogue of Riemann's existence theorem, the finite \'etale morphism  $V\to G^{\ad}$ algebraizes; see  \cite[Theorem.~3.1]{luk}.  Since $K$ is an algebraically closed field of characteristic zero and $H\to G$ is finite \'etale Galois, there is a   group structure on the connected $K$-scheme $H$ such that $H\to G$ is a homomorphism; see \cite{BrionSzamuely}. This proves the claim.

	 Since  $X$ is $K$-analytically Brody hyperbolic,  the induced morphism $H^{\ad}\to X^{\ad}$ is constant.   It follows readily that the morphism $\varphi:G^{\ad}\to Y^{\ad}$ is  constant.
	\end{proof}

\subsection{Testing hyperbolicity on algebraic groups with good reduction}
The main result of this section (Theorem \ref{thm:test_on_gr}) shows that the hyperbolicity of a variety over $K$ can be tested on \emph{analytic} maps from $\mathbb{G}_m$ and \emph{algebraic} maps from abelian varieties with good reduction. Our proof uses the (algebraic) classification theory of algebraic groups, GAGA for rigid analytic varieties, and   a uniformization theorem of Bosch--L\"utkebohmert; we note that   Cherry also used this  uniformization theorem to prove  the non-archimedean version of  Bloch's conjecture (Theorem \ref{thm:bloch}).     We start with two preliminary lemmas.

\begin{lemma}\label{lemma:structure_result} Let $X$ be an adic space over $K$. The following are equivalent.
\begin{enumerate}
\item The adic space $X$ is $K$-analytically Brody hyperbolic.
\item For every abelian variety $A$ over $K$, every morphism $A^{\ad}\to X^{\ad}$ is constant, and every morphism $\mathbb{G}_m^{\ad}\to X^{\ad}$ is constant. 
\end{enumerate}
\end{lemma}
\begin{proof} We follow the proof of   \cite[Lemma~2.4]{JKam}. It suffices to show that $(2)~\implies~(1)$.  Since $K$ is of characteristic zero,  a finite type connected group scheme over $K$ is   smooth, integral and quasi-projective over $K$ \cite[Tag~0BF6]{stacks-project}. In particular, by Chevalley's theorem on algebraic groups \cite{conradChev}, there is a unique affine normal subgroup scheme $H$ of $G$ such that $G/H$ is an abelian variety over $K$. Since $H$ is a smooth connected affine algebraic group, every two points lie on the image of a morphism of $K$-schemes $\mathbb A^1_{K}\setminus\{0\}\to H$.  Thus, as every morphism $\mathbb{G}_{m,K}^{\ad}\to X$ is constant (by $(2)$), we see that every morphism $H^{\ad}\to X$ is constant. In particular, every morphism $G^{\ad}\to X$ factors via a morphism $(G/H)^{\ad}\to X$. However, since $G/H$ is an abelian variety, we conclude that $(G/H)^{\ad}\to X$ is constant (by $(2)$). This concludes the proof.
\end{proof}

	\begin{lemma}\label{lem:gaga}  
		If $X$ is a   finite type separated scheme  over $K$ and  $A$ is an abelian variety over $K$, then every morphism $\varphi\colon A^{\ad}\to X^{\ad}$ is algebraic. 
		\end{lemma} 
		\begin{proof} 
		We first use Nagata's theorem to compactify $X$. Thus,  let  $\overline{X}$ be a proper scheme over $K$ and let $X\subset \overline{X}$ be an open immersion.  Now, by GAGA for proper schemes over $K$ \cite{conradGAGA, kopfGAGA}, the composed morphism $A^{\ad}\to X^{\ad} \to \overline{X}^{\ad}$  algebraizes.  In particular, the morphism $\varphi$ algebraizes.
	\end{proof}

	We now show that $K$-analytic Brody hyperbolicity can be tested on analytic maps from $\mathbb{G}_m^{\an}$ and algebraic maps from abelian varieties.
	\begin{proposition}\label{prop:obvious}
 	Let $X$ be a  finite type separated scheme over $K$. The following are equivalent.
	\begin{enumerate}
		\item $X$ is $K$-analytically Brody hyperbolic.
		\item For every abelian variety $A$ over $K$, every morphism $A\to X$ of schemes over $K$ is constant, and $X$ is $K$-analytically pure (Definition \ref{def:cherry_hyp}).
	\end{enumerate}
	\end{proposition}
	\begin{proof} Assume $(2)$ holds.   In order to prove that $X^{\ad}$ is $K$-analytically Brody hyperbolic, let $A^{\ad}\ra X^{\ad}$ be a morphism. By  Lemma \ref{lem:gaga},  this morphism is algebraic, and hence constant by $(2)$. Therefore,  by Lemma \ref{lemma:structure_result}, as $X$ is $K$-analytically pure,   we conclude that $X$ is $K$-analytically Brody. This shows that $(2)\implies (1)$. The other implication is straightforward.
	\end{proof}

 The uniformization theorem of Bosch--L\"utkebohmert that we require reads as follows.

 \begin{theorem}[Bosch--L\"utkebohmert]\label{thm:raynaud}  Let $K$ be a complete non-archimedean algebraically closed  field.
 	Let $A$ be an abelian variety over $K$. Then the following data exists.
 	\begin{enumerate}
 	\item A semi-abelian variety $G$ over $K$;
 	\item  An abelian variety $B$ over $K$ with good reduction over $\OO_K$;
 	\item  A surjective morphism of $K$-group schemes $G\to B$ whose kernel is a torus $\mathbb{G}_{m,K}^{\dim G - \dim B}$,   and
 	\item  A topological covering $G^{\an}\to A^{\an}$.
 	\end{enumerate}
 \end{theorem}
 \begin{proof}
 	See \cite[Theorem.~8.8~and~Remark.~8.9]{BLII}. 
 \end{proof}
 
 Recall that an abelian variety $A$ over $K$ has good reduction (over $\OO_K$) if there is an abelian scheme $\mathcal{A}$ over $\OO_K$ and an isomorphism of schemes $\mathcal{A}_K\cong A$ over $K$.
 The following result says, roughly speaking, that one can test the $K$-analytic hyperbolicity of a variety over $K$ on analytic tori and on abelian varieties with good reduction.
 
 \begin{theorem}\label{thm:test_on_gr}  
 Let $X$ be a finite type separated scheme over $K$. Then the following are equivalent.
 \begin{enumerate}
 \item   $X$ is $K$-analytically Brody hyperbolic.
 \item Every morphism   $\mathbb{G}_{m,K}^{\an} \to X$ is constant and, for every abelian variety $B$ over $K$ with good reduction over $\OO_K$, every morphism      $B\to X$ is constant.
 \end{enumerate}
 \end{theorem}
\begin{proof}
Clearly, $(1)\implies (2)$. To prove the theorem, assume that $(2)$ holds. In particular, $X$ is $K$-analytically pure.  
	 Let $A$ be an abelian variety over $K$ and let $A^{\an}\to X_K^{\an}$ be a morphism. By Lemma \ref{lemma:structure_result}, it sufices to show that this morphism  is constant.
	
	By Theorem \ref{thm:raynaud}, there is a  semi-abelian variety $G$ over $K$, an abelian variety $B$ over $K$ with good reduction over $\OO_K$, a surjective homomorphism $G\to B$ whose kernel $T$ is a torus, and a topological covering $G^{\an}\to A^{\an}$.   
	
	Since every morphism $\mathbb{G}_m^{\an}\to X_K^{\an}$ is constant, the  composed morphism 
	\[
	T^{\an}\subset G^{\an}\to A^{\an} \to X_K^{\an}
	\] is constant. Therefore, the composed morphism $G^{\an}\to A^{\an}\to X_K^{\an}$  factors over a morphism $B^{\an}=G^{\an}/T^{\an}  \to X_K^{\an}$.  
	By GAGA (Lemma \ref{lem:gaga}), the morphism $B^{\an}\to X^{\an}$ is the analytification of a morphism $B\to X$.  By assumption $(2)$, this  morphism $B\to X$ is constant. We conclude that $B^{\an}\to X^{\an}$ and thus $A^{\an}\to X^{\an}$ is constant, as required. 
\end{proof}

We stress that Theorem \ref{thm:test_on_gr} will be used to show that the generic fiber of a proper scheme over $\OO_K$ ``inherits'' $K$-analytic  Brody hyperbolicity from the ``grouplessness'' of its  special fiber (Theorem \ref{thm:main_result}).

	 \subsection{Inherting   hyperbolicity from the special fiber}\label{section:inherit}
Given a proper finitely presented scheme $X$ over $\OO_K$, the ``hyperbolicity'' of its special fiber forces the ``hyperbolicity'' of its generic fiber. We will make this more precise in the next section (see  Theorem \ref{thm:main_result}). In this section, we prove some preliminary results necessary to prove the latter result.

	 We start with a     criterion for a   map of topological  spaces to be constant. We recall first some definitions about spectral spaces; see \cite{Hochster} or \cite[Section~08YF]{stacks-project}.     
	 
	 \begin{definition}
A topological space $X$ is \emph{spectral} if $X$ is quasi-compact, if $X$ has a
basis of quasi-compact open subsets which is stable under finite intersections, and if every
irreducible closed subset of $X$ is the closure of a unique point. 
	 \end{definition} 
	 
	  We recall that examples of spectral spaces are given by quasi-compact quasi-separated schemes and quasi-compact quasi-separated adic spaces over $K$ \cite[Lemma 1.3.12]{huber}.  We will   apply the next  lemma to a situation where both  these types of spaces appear.
	
 	 \begin{lemma} \label{lem:topology}
 	 	Let $f\colon X\ra Y$ be a quasi-compact continuous map between spectral spaces, with $X$ connected and $Y$ noetherian. Let $N$ be the set of non-closed points in the image. If $N$ is empty, then $f$ is constant.  Otherwise, the image of $f$ is included in the closure of $N$.
 	 \end{lemma}
 
\begin{proof} Let $Z$ be the closure of $N$ and suppose that there is a (closed) point  $y$ in the image of $f$ which does not lie in $Z$. The set $U\colonequals f^{-1}(Y\setminus\{y\})$ is open and quasi-compact.  Let $x'\in U$ be a point and let  $x$ be a  point in the closure of  $x'$, so that $f(x)$  is in the closure of $f(x')$. If $f(x')$ lies in $Z$, then $f(x)$   also lies in $Z$. Otherwise, the point $f(x')$ is closed, so  that $f(x)=f(x')$. In both cases, we deduce that $x'$ lies in $U$. Therefore $U$ is stable under specializations, and is therefore  closed \cite[Tag~0903]{stacks-project}. As $X$ is connected, we conclude that $U$ is empty, as desired. 
\end{proof}

	 \begin{corollary}\label{lem:reduction1}\label{cor:crit_for_constancy}
	 Let $\mfX$ be  a quasi-compact  separated formal scheme over $\OO_K$ which is topologically of finite presentation, and let $\pi:\mfX_\eta\to \mfX_k$ be the specialization map. Let $S$ be a connected  finite type separated scheme  over $K$. A map of rigid analytic varieties $f\colon S^{\an}\ra\mfX_\eta$ is constant if and only if every point in the image of the composite continuous map $\pi\circ f$  is closed.
	 \end{corollary}
	 \begin{proof} Suppose that every point in the image $\pi\circ f$ is closed.
	We claim that the composite map $\pi\circ f$ is constant. To prove this, we can restrict $S^{\an}$ to some quasi-compact connected open subset $T$. Since the specialization map is quasi-compact \cite[(0.2.2)-(0.2.3)]{Berth} and $f$ is quasi-compact,   the composite $\pi\circ f\colon T\ra\mfX_\sigma$ is a quasi-compact continuous map of spectral spaces with a noetherian target. The claim that $\pi\circ f$ is constant now follows from the previous lemma.
	
	We conclude that there is a   point $x$ in $\mfX_\sigma$ such that $f$ factors over the set $\pi^{-1}(x)$. Note that $\pi^{-1}(x)$ is included in $\mfU_\eta=\pi^{-1}(\mfU_\sigma)$ for some open affine subscheme $\mfU $ of $\mfX$. Since $\mfU_\eta$ is affinoid \cite[(0.2.2.1)]{Berth}, the corollary follows from the hyperbolicity of affinoids (Proposition \ref{prop:aff_are_brody}).
	 \end{proof}

  We recall that whenever one considers a  formal scheme $\mfX$ over $\OO_K$ which is obtained as the formal completion of a proper finitely presented scheme $X$ over $\OO_K$, its generic fiber $\mfX_\eta$ coincides with the $K$-analytic space associated to $X_K$ \cite[Proposition (0.3.5)]{Berth}.

We now investigate the ``generic'' properties of a  proper scheme over $\OO_K$ whose special fiber  does not contain any curves of genus at most $g$.  To make this more precise, we start with     two definitions.
  \begin{definition}\label{defn:no_curves} If $X$ is a proper scheme over  a field $k$, then we say that $X$ \emph{has no   curves of genus at  most $g$} if, for every smooth projective geometrically connected curve $C$ over $k$ of genus $\leq g$, every morphism $C\to X$ is constant. 
  \end{definition}
  \begin{definition}  If $X$ is a rigid analytic variety over $K$, we say that $X$ \emph{has no $K$-analytic curves of genus at most $g$} if, for every smooth projective geometrically connected curve $C$ of genus $\leq g$ and every dense open subscheme $\widetilde{C}\subset C$, every morphism $\widetilde{C}^{\an}\to X$ is constant.
  \end{definition}

  If $X$ is a rigid analytic variety over $K$ with no $K$-analytic curves of genus zero, then  $X$ is $K$-analytically pure (Definition \ref{def:cherry_hyp}).
  
      The following results are inspired by (and generalize) \cite[Lemmas~2.12-2.13]{CherryKoba}. The proof in \emph{loc. cit.} uses  the Berkovich topology. 
  
  \begin{proposition}\label{prop:constant}
Let $X$ be a proper  scheme  over $\OO_K$ such that $X_k$ has no curves of genus at most $g$. Let $C$ be a smooth projective connected curve of genus at most $g$ over $K$ and let $U\subset C^{\an}$ be a connected open analytic subvariety of $C^{\an}$.  If $f\colon U\ra X_K^{\ad}$ is a $K$-analytic map, then the composite map of topological spaces $U\ra X_K^{\ad}\ra X_k$ is constant.
\end{proposition}
\begin{proof}
By Lemma \ref{lem:topology}, it suffices to show that  every point in the image of $C^{\ad}\to {X_k}$ is closed. 

 Let $\tilde{x}$ be a non-closed point of ${X_k}$ with local field $\kappa(\tilde{x})$.  Suppose that there is a point $c$ in $C^{\ad}$ with $\pi\circ f(c) = \tilde{x}$. Define $x:=f(c)$. Let $\kappa(c)$ [resp. $\kappa(x)$] be the valuation field associated to $c$ [resp. $x$] and let $\widetilde{\kappa(c)}$ [resp. $\widetilde{\kappa(x)}$] be its residue field.  Then we have an inclusion of fields
	 \[
	 \kappa(\tilde{x}) \subseteq \widetilde{\kappa(x)} \subseteq \widetilde{\kappa(c)}.
	 \] Since $c\in C^{\ad}$, we have  that $\widetilde{\kappa(c)} = k$ or $\widetilde{\kappa(c)} $ is the function field of smooth projective connected curve over $k$  of genus at most $g$ \cite[Remark 4.18]{BPR}.  Since $\tilde{x}$ is a non-closed point of ${X_k}$ and ${X_k}$ is a finite type $k$-scheme, we have that $\kappa(\tilde{x}) \neq k$. Thus, the curve $C_{\tilde{x}}\subset {X_k}$ given by the closure of $\tilde{x}$ in ${X_k}$ is dominated by a curve of genus at most $g$.   This contradicts the fact that   ${X_k}$ admits no curves of genus at most $g$. Thus, the point $\tilde{x}$ is not contained in the image of $\pi\circ f:C^{\ad}\to {X_k}$, i.e., every point in the image is closed. 
\end{proof}

\begin{corollary}\label{thm:no_curves} Let $X$ be a proper  scheme   over $\OO_K$ such that ${X}_k$ has no smooth projective connected curves of genus at most $g$. Then ${X}_K^{\ad}$ has no $K$-analytic curves of genus at most $g$.
\end{corollary}
\begin{proof} 
Combine Corollary \ref{cor:crit_for_constancy} and  Proposition \ref{prop:constant}.
\end{proof}

 \section{Grouplessness, purity, and their relation with    non-archimedean hyperbolicity}\label{sec:rel} 
 In this section we investigate some notions of hyperbolicity for algebraic varieties. In particular, we verify some predictions made by Green--Griffiths--Lang's conjecture. To prove our results, we will  combine the non-archimedean results  with  the uniformization  theorem of Bosch--L\"utkebohmert  \cite{BLII}.  To simplify some of our proofs and statements, we introduce some terminology.

 \begin{definition}\label{def:groupless} Let $k$ be an algebraically closed field. Let $X$ be a finite type   scheme over $k$.
 	\begin{enumerate}
\item  We say $X$ is \emph{groupless over $k$} if, for any finite type connected group scheme $G$ over $k$, every morphism $G\to X$ is constant.
\item We say $X$ is \emph{pure over $k$} if, for every normal variety $T$ over $k$ and every dense open $U\subset T$ with $\textrm{codim}(T\setminus U)\geq 2$, we have that every morphism $U\to X$ extends (uniquely) to a morphism $T\to X_{k}$.
	\end{enumerate}
 \end{definition}

Note that Kov\'acs \cite{KovacsSubs}, Kobayashi \cite[Remark~3.2.24]{Kobayashi} and Hu--Meng--Zhang \cite{ShangZhang} refer to groupless varieties (over $\mathbb{C}$) as being ``algebraically hyperbolic'' or ``algebraically Lang hyperbolic'' or ``algebraically Brody hyperbolic''. We  avoid this  terminology, as ``algebraic hyperbolicity'' more commonly refers to   a notion introduced   by Demailly \cite{BKV, Demailly, JKam}.

Note that Proposition \ref{prop:obvious} says that the non-archimedean analogue  (Conjecture \ref{conj:lang}) of the Green--Griffiths--Lang conjecture boils down to showing the non-existence of ``analytic'' tori mapping to a groupless variety. Thus, we see that Conjecture \ref{conj:lang} is equivalent to the following conjecture. 
	
\begin{conjecture}[Non-archimedean Green--Griffiths--Lang Conjecture]  
Let $X$ be a projective variety over $K$. If $X$ is groupless over $K$, then $X$ is $K$-analytically pure over $K$.
\end{conjecture}
	
Needless to stress, a smooth  projective $K$-analytically pure variety is not necessarily $K$-analytically Brody hyperbolic. Indeed, a smooth proper connected genus one curve      over $K$ with good reduction over $\OO_K$ is $K$-analytically pure \cite[Theorem~3.6]{Cherry}, but (clearly) not groupless (and therefore not $K$-analytically Brody hyperbolic).

\subsection{Properties of groupless and pure varieties}
Groupless and pure varieties are studied systematically in \cite{JKam, JXie}. Some of the   properties of groupless varieties we need in this paper  are summarized in the following two remarks.

\begin{remark}\label{remark:gr_and_pure}
If $k$ is an algebraically closed field of characteristic zero and $X$ is a proper scheme over $k$, then $X$ is groupless (over $k$) if and only if for every abelian variety $A$ over $k$ every morphism $A\to X$ is constant \cite[Lemma~2.5]{JKam}. Moreover, a proper scheme  $X$ over $k$ is pure if and only if it has no rational curves, i.e., every morphism $\mathbb{P}^1_k\to X$ is constant \cite[Lemma~3.5]{JKam}. In other words,  
  a proper scheme over $k$ is pure if and only if it contains no curves of genus at most zero in the sense of Definition \ref{defn:no_curves}.  In particular, if $X$ is a   proper groupless   variety   over $k$, then $X$ is pure over $k$. The latter statement also follows from \cite[Proposition~6.2]{GLL}.
\end{remark}

\begin{remark}\label{remark:geometricity_groupless}   Let $k$ be an algebraically closed field of characteristic zero, and let $k\subset L$ be an algebraically closed field extension.
Let $X$ be a pure (resp. groupless) variety over $k$. Then $X_L$ is pure (resp. groupless) over $L$. This is proven in \cite{JKam}.  
\end{remark}

 \begin{definition}
 Let $k$ be a field and let $X$ be a finite type scheme over $k$. We say that $X$ is \emph{groupless over $k$} (respectively \emph{pure over $k$}) if there is an algebraically closed field extension $k\subset \Omega$ such that $X_{\Omega}$ is groupless over $\Omega$ (respectively pure over $\Omega$).  In this case, $X$ is groupless over any field extension of $k$. This is a well-defined property of $X$ over $k$ by  Remark \ref{remark:geometricity_groupless}.
 \end{definition}

We will   need the following application of Zarhin's trick for abelian varieties.

\begin{lemma}\label{lem:zarhin}
Let $X$ be a proper scheme over $k$. Then $X$ is groupless over $k$ if and only if, for every principally polarizable abelian variety $B$ over $k$, every morphism $B\to X$ is constant.
\end{lemma}
\begin{proof}
Let $A$ be an abelian variety over $k$, and let $A\to X$ be a morphism. To prove the lemma, as $X$ is proper over $k$, it suffices to show that $A\to X$ is constant (Remark \ref{remark:gr_and_pure}). To do so, let $B:= A^4 \times A^{\vee, 4}$, where $A^\vee$ is the dual abelian variety of $A$. Note that $B$ is principally polarizable by Zarhin's trick \cite{Zarhin}. Therefore, by assumption, the composed morphism $B\to A\to X$ is constant, where $B\to A$ is the projection onto the first coordinate. Since $B\to A$ is surjective and $B\to A\to X$ is constant, we conclude that $A\to X$ is constant, as required.
\end{proof}

\subsection{Hyperbolicity of the special fiber and the analytic generic fiber}
Given a  proper finitely presented scheme over $\OO_K$, we now    relate the non-archimedean hyperbolicity  of its generic fiber to the ``algebraic'' hyperbolicity of its special fiber. We begin with the following immediate consequence of  Corollary \ref{thm:no_curves}.

\begin{proposition}\label{prop:pur_gen_baby} Let $X$ be a proper finitely presented scheme over $\OO_K$ with a pure special fiber over $k$. Then, for any rational curve $C$, every morphism $C^{\an}\to X_K^{\an}$ is constant. In particular,  $X_K$ is  $K$-analytically pure. 
\end{proposition}
\begin{proof}
Since the special fiber of $X\to \Spec \OO_K$ is pure and proper, it has no rational curves (Remark \ref{remark:gr_and_pure}). Thus, the special fiber has no curves of genus at most zero. Therefore, by Corollary \ref{thm:no_curves}, the generic fiber $X_K$ has no $K$-analytic curves of genus at most zero. This concludes the proof.
\end{proof}

We now prove that purity of one fiber in a family of varieties is inherited by any fiber that specializes to it.  

\begin{corollary} \label{cor:purity_gens}
Let $X\to S$ be a  proper  morphism of schemes. with $S$ an integral regular noetherian local scheme. Let $s$ be the unique  closed point of $S$, and let   $s'\in S$ be a point.  If $X_{s}$ is   pure over $k(s)$, then  $X_{s'}$ is   pure over $k(s')$.
\end{corollary}

\begin{proof} 
	Write $S=\Spec A$. If $\dim S = 0$, then the statement is clear. Thus, we may and do assume that $d:=\dim S \geq 1$. We now proceed by induction on $d$.
	
	If $d=1$, we may and do assume that $s'$ is the generic point of $S$. Let $K$ be a complete algebraic closure of $\mathrm{Frac}(A)$ and note that $K$ is naturally endowed a valuation. Let  $\OO_K$ be the associated valuation ring and write $T:=\Spec \OO_K$. Note  that $X_T \to T$ is a proper morphism and that its special fiber is   pure. By Proposition \ref{prop:pur_gen_baby}, it follows that $X_K$ is $K$-analytically Brody hyperbolic, so that $X_{s'}$ is   pure. This proves the statement when $d=1$.  
	
	Assume that $d>1$.  Let $(x_1,\ldots,x_d)$ be a minimal set of generators for the maximal ideal of $A$. Note that $(x_1,\ldots,x_d)$ is a regular sequence and that $B:=A/(x_1)$ is a regular ring    ring of dimension $d-1$; see \cite[Tag~00NQ]{stacks-project}. Now, as the special fiber of $X_B \to \Spec B$ is isomorphic to the special fiber of $X\to S$, it follows from the induction hypothesis that every fiber of $X_B\to B$ is    pure.   Note that $\Spec B \subset S$ is a closed subscheme. If $s'\in \Spec B$, then we are done. Otherwise, we have that $s'\in \Spec A_{x_1}$. Note that the special fiber of $X_{A_{x_1}}\to \Spec A_{x_1}$ is the generic fiber of $X_B\to \Spec B$, and is thus pure. Therefore,  since $\dim A_{x_1} <d $, by applying the induction hypothesis to $X_{A_{x_1}}\to \Spec A_{x_1}$, we conclude that $X_{s'}$ is   pure. 	
\end{proof}

Similarly to Proposition \ref{prop:pur_gen_baby},  the main result of this section (Theorem \ref{thm:main_result})  says that the grouplessness of the special fiber implies  the analytic hyperbolicity of the generic fiber. To prove our main result, we will use  Theorem \ref{thm:test_on_gr} and   the following   ``algebraic'' lemmas. These lemmas will help us reduce to the case of ``good reduction'' in the proof of our main result (Theorem \ref{thm:main_result}).
 
 \begin{lemma}\label{lem:gll} Let $S$ be an integral regular noetherian scheme.
Let $X\to S$ be a proper morphism of schemes. Assume that the geometric fibers of $X\to S$ contain no rational curves. Let $Z$ be an integral smooth finite type scheme over $S$,   and let $U\subset Z$ be a dense open.
Then, every $S$-morphism $U\to X$  extends to a morphism $Z\to X$.
 \end{lemma}
 \begin{proof} Let $K$ be the function field of $Z$. The morphism $U\to X$ induces a $K$-section of the morphism $X\times_S Z \to Z$. For any geometric point $z$ in $Z$, its   fiber $X_z\otimes k(z)$  contains no rational curves. Therefore, by \cite[Proposition~6.2]{GLL}, the morphism $X\times_S Z\to Z$ has a section which is compatible with the given morphism $U\to X$. 
 \end{proof}
 
 \begin{lemma}\label{lem:good_reduction_case} Let $S$ be an integral regular noetherian scheme with $K=K(S)$. Let $\mathcal{B}\to S$ be an abelian scheme and let ${X}\to S$ be a proper morphism whose geometric fibers are pure.
 	Let $\mathcal{B}_K\to {X}_K$ be a non-constant morphism. Then $\varphi$ extends uniquely to a morphism $\mathcal{B}\to {X}$.  
 \end{lemma}
 \begin{proof}
 	Since $X\to S$ is proper, by the valuative criterion of properness, there is a dense open $U\subset \mathcal{B}$ whose complement is of codimension at least two and a morphism 
 	 $U\to X$ which extends the morphism $\mathcal{B}_K\to {X}_K$.  By Lemma \ref{lem:gll},  the morphism $U\to {X}$ extends to a morphism $\mathcal{B}\to X$.
 \end{proof}
 
 \begin{lemma}\label{lem:descend}
 Let $\OO$ be a valuation ring with algebraically closed fraction field $K$ and let $S = \Spec \OO$. Let $X\to S$ be a finitely presented morphism of schemes. Then there is an integral regular finite type  scheme $T$ over $\ZZ$, a dominant morphism $S\to T$, and a finitely presented morphism $\mathcal{X}\to T$ with $\mathcal{X}\times_T S\cong X$ over $S$.
 \end{lemma}
 \begin{proof} Since $X\to S$ is finitely presented, we may descend $X\to S$ to a morphism of finite type schemes over $\ZZ$. Thus,  let $T_1$ be an   integral finite type $\ZZ$-scheme, let $\mathcal{X}_1\to T_1$ be a finitely presented morphism, let $S\to T_1$ be a dominant morphism, and let $\mathcal{X}_1\times_{T_1} S \cong X$ be an isomorphism of schemes over $S$. Now, we use an alteration to ``resolve'' the singularities of $T_1$; see \cite[Theorem~8.2]{deJongAlt}. Thus, let $T\to T_1$ be proper surjective generically finite   morphism with $T$ an integral regular   finite type scheme over $\ZZ$. Since $K$ is algebraically closed, the point $\Spec K\to S \to T_1$ factors over $T\to T_1$. Then, since $\OO$ is a valuation ring, it follows from the valuative criterion of properness \cite[Tag~0A40]{stacks-project} that the dominant morphism $S\to T_1$ factors over $T\to T_1$. Let $\mathcal{X}:= \mathcal{X}_1\times_{T_1} T$. Then, $\mathcal{X}\to T$ is a finitely presented morphism over an integral regular finite type scheme $T$ over $\ZZ$ with $\mathcal{X}\times_T S\cong X$ over $S$, as required. 
 \end{proof}

 \begin{lemma}\label{lem:good_reduction_case_1} 
 Let $\OO$ be a valuation ring with algebraically closed fraction field $K$ and let $S = \Spec \OO$.  Let $ B\to S$ be an abelian scheme and let $X\to S$ be a proper morphism whose geometric fibers are pure.
 	Let $ {B}_K\to {X}_K$ be a non-constant morphism. Then $\varphi$ extends uniquely to a morphism $ {B}\to {X}$. The morphism $\mathcal{B}_K\to {X}_K$ is constant if and only if  $\mathcal{B}_k\to {X}_k$ is constant. 
 \end{lemma}
 \begin{proof}   We first descend  ``everything'' to an integral regular noetherian base $T$   using Lemma \ref{lem:descend}. Thus, 
 let $T$ be an integral regular finite type affine $\ZZ$-scheme, let $S\to T$ be a dominant morphism, let $\mathcal{X}\to T$ be a proper finitely presented morphism with $\mathcal{X}_S \cong X$ over $S$, let $\mathcal{B} \to T$ be an abelian scheme with $\mathcal{B}_S \cong B $ over $S$, and let $\mathcal{B}_{K(T)} \to \mathcal{X}_{K(T)}$ be a morphism which agrees with $B_K \to X_K$ after base-change along $K\subset K(T)$. 
 
 Let $t\in T$ be the image of the (unique) closed points of $S$ in $T$.
Define $T' = \Spec \mathcal{O}_{T,t}$ to be the spectrum of the local ring at $t$.  Note that $T$ is an integral regular noetherian scheme, and that the morphism $S\to T $ factors over the natural morphism $T'\to T$. Indeed, write $T=\Spec A$ and let $\mathfrak{p} = \varphi^\ast \pi$, where $\pi\subset \OO$ is the unique maximal ideal of $\OO$ and $\varphi:A\to \OO$ is the (injective) morphism associated to $S\to T$. The morphism $\varphi:A\to \OO$ induces a morphism $A_{\mathfrak{p}}\to K$. However, the image of this morphism is contained inside $\OO$. Indeed, the ring $A_{\mathfrak{p}}$ is mapped to $\OO_{\mathfrak{m}}$, but the latter equals $\OO$.

Note that the geometric fibers of $X\to S$ are pure (by assumption). Therefore, the fiber of $\mathcal{X}\to T'$ over the (unique) closed point is pure. Thus, by Corollary \ref{cor:purity_gens},  every geometric fiber of $\mathcal{X}\to T'$ is pure. 
  Therefore,  
 	 by Lemma \ref{lem:good_reduction_case}, the morphism $\mathcal{B}_K\to \mathcal{X}_K$  extends uniquely to  a morphism $\mathcal{B}_{T'}\to \mathcal{X}_{T'}$ over $T'$. The morphism $B=\mathcal{B}_S := \mathcal{B}\times_{T'} S \to \mathcal{X}_S = X$ clearly extends the morphism $B_K\to X_K$. This concludes the proof of the first statement of the lemma.
 	 
 	 To prove the second statement, let $Y\subset X$ be the reduced scheme-theoretic image of the morphism $\mathcal{B}\to X$. Note that $Y$ is a closed subscheme of $X$.  Since $\mathcal{B}\to S$ is an abelian scheme over $S$, the morphism $\mathcal{B}\to S$ is surjective. Therefore, the morphism $Y\to S$ is surjective. The fibers of $Y\to S$ are of the same dimension \cite[Tag~00QK]{stacks-project}.   This concludes the proof.
 \end{proof}

\begin{theorem}\label{thm:main_result}  
	Let   ${X}$ be a proper  finitely presented scheme over  $\Spec \OO_K$. If the special fiber $X_k$ is   groupless, then the generic fiber ${X}_K$ is $K$-analytically Brody hyperbolic. 
\end{theorem}
\begin{proof}  
	Since the special fiber is groupless and proper over $k$, it is pure over $k$ (Remark \ref{remark:gr_and_pure}). Therefore, by Proposition \ref{prop:pur_gen_baby},  the variety ${X}_K$ is $K$-analytically pure. To conclude the proof, let $B$ be  an abelian variety over $K$ with good reduction over $\OO_K$  and   let $B \to X$ be a morphism. To prove the theorem, as we can test hyperbolicity on algebraic groups with good reduction (Theorem \ref{thm:test_on_gr}), it suffices to show that this morphism is constant.

	Let $S= \Spec \OO_K$ and let $\mathcal{B}\to S$ be an abelian scheme over $S$ with $\mathcal{B}_K \cong B$.  Since  the special fiber of ${X}$ is  pure, the first part of Lemma \ref{lem:good_reduction_case_1} implies that the morphism $B\to X$ extends to a morphism $\mathcal{B}\to {X}$. Since the special fiber ${X}_k$ is groupless, the induced morphism   $\mathcal{B}_k\to {X}_k$ is constant. Therefore, by the second part of Lemma \ref{lem:good_reduction_case_1}, the morphism $B\to X_K$ is constant, as required. 
\end{proof}

 \begin{proof}[Proof of Theorem  \ref{thm:gen_fiber_is_hyp}]
 This is Theorem \ref{thm:main_result}.
 \end{proof}

	\subsection{Evidence for the non-archimedean Green--Griffiths--Lang conjecture}
	As we have already mentioned in the introduction, in light of the Green--Griffiths--Lang conjecture, it seems reasonable to suspect that a   projective groupless variety over $K$ is  $K$-analytically Brody hyperbolic; see Conjecture \ref{conj:lang}.  In this direction, we first prove the following result for ``constant'' varieties.

	\begin{corollary}\label{cor:char}	Let $X$ be a proper  scheme over an algebraically closed field $k$. Fix a complete  algebraically closed non-archimedean valued field $K$ endowed with a section $k\hookrightarrow\OO_K$ of the quotient map.  
		\begin{enumerate}
			\item $X$ is pure over $k$ if and only if  $ X_K^{\an}$ is $K$-analytically pure.
			\item $X$ is groupless over $k$ if and only if  $ X_K^{\an}$ is $K$-analytically Brody hyperbolic.
		\end{enumerate}
	\end{corollary}
	
	\begin{proof}
		If $X_k$ is pure (respectively groupless) then $X_K^{\ad}$ is $K$-analytically pure (respectively $K$-analytically Brody hyperbolic)    by Proposition \ref{prop:pur_gen_baby} (respectively Theorem \ref{thm:main_result}). Conversely, a non-constant map $A_k\ra X_k$ from an abelian variety $A$ over $k$ induces a non-constant map  $A_K\ra X_K$. 
	\end{proof}
	
\begin{proof}[Proof of Theorem \ref{thm:lang_for_constant}]
This follows from the second part of Corollary \ref{cor:char}.
\end{proof}

Cherry's theorem (Theorem \ref{thm:bloch}) says that a closed groupless subvariety of an abelian variety over $K$ is $K$-analytically Brody hyperbolic, and thereby confirms that the non-archimedean Green--Griffiths--Lang conjecture holds for closed subvarieties of abelian varieties.
We now prove    the following two  consequences of Cherry's work.  
	\begin{proposition}\label{prop:lv_for_curves}   
	Let $X$ be a quasi-projective integral curve over $K$. Then $X$ is groupless if and only if $X$ is $K$-analytically Brody hyperbolic.
	\end{proposition}
	\begin{proof} Assume $X$ is groupless.  Since $X$ is a curve, to prove that $X$ is $K$-analytically  Brody hyperbolic, we may and do assume that $X$ is smooth. Then, as $X$ is a groupless smooth quasi-projective curve over $K$, there is a finite \'etale cover $Y\to X$ of $X$ such that the smooth projective model $\overline{Y}$ of $Y$ is of genus at least two. In particular, $\overline{Y}$ is $K$-analytically Brody hyperbolic (Theorem \ref{thm:bloch}). Therefore, the open subset $Y$ of $\overline{Y}$ is $K$-analytically Brody hyperbolic. Since $Y\to X$ is finite \'etale and hyperbolicity descends along finite \'etale maps (Proposition \ref{prop:chevalley_weil}), we conclude that $X$ is $K$-analytically Brody hyperbolic, as required.
	\end{proof} 
	
Following ``standard'' arguments (see for instance \cite{LevinAaron, vanderGeer}), we now   prove the non-archimedean version of the Green--Griffiths--Lang conjecture for symmetric powers of smooth projective curves.
	
	\begin{proposition}[Non-archimedean Green--Griffiths--Lang conjecture for symmetric powers of projective curves]\label{prop:lang_for_syms}
	Let $d\geq 1$ be an integer, and let $X$ be a smooth projective  integral curve over $K$. Then $\mathrm{Sym}^d_{X}$ is groupless if and only if $\mathrm{Sym}^d_{X}$ is $K$-analytically Brody hyperbolic.
	\end{proposition} 
	\begin{proof}We assume that $\mathrm{Sym}^d_{X}$ is groupless (the other implication is obvious). Let $J $ be the Jacobian of $X$ over $K$. Fix a point $P$ in $X(K)$.
	Consider the morphism $\mathrm{Sym}^d_X\to J$ given by $[x_1,\ldots,x_d] \mapsto x_1+\ldots+x_d - dP$ in $J$.  The fibers of this morphism are projective spaces \cite[Tag~0CCT]{stacks-project}. Thus, since $\mathrm{Sym}^d_X$ is groupless, this morphism is injective.  Let $W^d$ be its image in $J$, and note that the subvariety $W^d\subset J$ is isomorphic to $\mathrm{Sym}^d_X$. In particular, the closed subvariety  $W^d$ of $J$ is a groupless closed subvariety. Therefore, by Cherry's theorem (Theorem \ref{thm:bloch}), we conclude that $W^d$ is $K$-analytically Brody hyperbolic, so that $\mathrm{Sym}^d_X$ is $K$-analytically Brody hyperbolic, as required.
	\end{proof}
 
	\begin{remark}
	Let $A$ be a simple abelian surface over $K$ with   good reduction over $\OO_K$. Let $X:=A\setminus \{0\}$. Since $X$ is simple, we see that $X$ is groupless. Interestingly, if $K=\CC$, then the variety $X$ is not Brody hyperbolic. However,  if $K$ is non-archimedean, the groupless variety $X$ is $K$-analytically Brody hyperbolic. Indeed, since $A$ has good reduction,  any morphism   $\mathbb{G}_m^{\an}\to A\setminus\{0\}\subset A$ is constant \cite[Theorem~3.6]{Cherry}, so that $X$ is $K$-analytically Brody hyperbolic by Proposition \ref{prop:obvious}.  
	\end{remark}

	\subsection{Application to the moduli space of abelian varieties}\label{section:app_to_mod}

  	For  an integer $N\geq 3$ and integer $g\geq 1$, let $\mathcal A_{g}^{[N]}$ be the fine moduli space of principally polarized abelian varieties over $\ZZ[1/N]$ with full level $N$-structure. Note that $\mathcal{A}_{g}^{[N]}$ is a smooth quasi-projective (non-proper)   scheme over $\ZZ[1/N]$; see \cite{MoretBailly}. 
 We let $\mathcal A_{g,K}^{[N],*}$ be the Satake-Baily-Borel compactification of $X$ over $K$.

 We will show that in the equicharacteristic zero case, the smooth quasi-projective scheme $\mathcal{A}_{g,K}^{[N]}$ is $K$-analytically Brody hyperbolic; see Corollary \ref{cor:Ag_is_K_hyp}. To do so, we will combine Nadel's  theorem (over the complex numbers) with the results proven in this paper. In fact, our proof of the $K$-analytic hyperbolicity of  $\mathcal{A}_{g,K}^{[N]}$ given below  is modeled on the proof of Proposition \ref{prop:lv_for_curves}.
        
        \begin{theorem}[Nadel]\label{thm:nadel_C}
        	For every integer $g$, there is an integer $\ell > 3$  such that $\mathcal{A}^{[\ell]*}_{g,\CC} $ is Brody hyperbolic (and groupless over $\CC$).
        \end{theorem}
        \begin{proof}
        We refer the reader to Nadel's paper \cite{Nadel} for the first statement. The fact that Brody hyperbolic varieties over $\CC$ are groupless over $\CC$ is well-known.  This proves the theorem.
        \end{proof}

  \begin{theorem}\label{thm:nadel_Cp} Fix a complete  algebraically closed non-archimedean valued field $K$ of characteristic zero endowed with a section $k\hookrightarrow\OO_K$ of the quotient map.
  For every integer $g$, there is an integer $\ell$ such that $\mathcal{A}_{g,K}^{[\ell],\ast}$ is $K$-analytically Brody hyperbolic.  
  \end{theorem}
  \begin{proof} We first apply Nadel's theorem (Theorem \ref{thm:nadel_C}). Thus, let $N>3$ be an integer such that $\mathcal{A}_{g,\CC}^{[\ell],\ast}$ is groupless over $\CC$. In particular, by the Lefschetz principle, the proper scheme $\mathcal{A}_{g,k}^{[\ell],\ast }$ is groupless over $k$. Therefore, as $\OO_K$ is endowed with a section $k\hookrightarrow \OO_K$ of the quotient map, it follows from  Corollary \ref{cor:char} that $\mathcal{A}_{g,K}^{[\ell],\ast}$ is $K$-analytically Brody hyperbolic.  
  \end{proof}

\begin{corollary} \label{cor:Ag_is_K_hyp} Fix a complete  algebraically closed non-archimedean valued field $K$ of characteristic zero endowed with a section $k\hookrightarrow\OO_K$ of the quotient map. For every integer $N\geq 3$, and for every integer $g$,  the quasi-projective scheme $\mathcal{A}_{g,K}^{[N]}$ is $K$-analytically Brody hyperbolic.  
\end{corollary}
\begin{proof}  By Theorem \ref{thm:nadel_Cp}, there is an integer  $\ell>1$   such that $\mathcal{A}_{g,K}^{[N\ell],\ast}$ is $K$-analytically Brody hyperbolic. Since $\mathcal{A}_{g,K}^{[N\ell]}$ is an open subscheme of $\mathcal{A}_{g,K}^{[N\ell],\ast}$, it follows that $\mathcal{A}_{g,K}^{[N\ell]}$ is $K$-analytically Brody hyperbolic. Since $\mathcal{A}_{g,K}^{[N\ell]}\to \mathcal{A}_{g,K}^{[N]}$ is finite \'etale and $K$-analytic hyperbolicity descends along finite \'etale maps (Proposition \ref{prop:cv_stack}), we conclude that $\mathcal{A}_{g,K}^{[N]}$ is $K$-analytically Brody hyperbolic.
\end{proof}

\subsection{Cherry's semi-distance}\label{section:distance}
In his thesis, Cherry  introduces and studies a natural analogue of Kobayashi's pseudometric in the non-archimedean setting. In this section we gather three new observations  on Cherry's semi-distance. The main result is that Cherry's semi-distance can ``only'' detect   rational curves for ``constant'' varieties (Theorem \ref{thm:distance}), and therefore fails to detect the hyperbolicity of the space in general. We refer to   \cite{Vasquez} for related results.

\begin{definition}
Let $X$ be a rigid analytic variety over an algebraically closed complete valued field $K$.
\begin{enumerate}
\item Let $x,y\in X(K)$. A \emph{Kobayashi chain joining $x$ and $y$} is a finite sequence of analytic maps$$
f_j\colon\mathbb{B}^1\ra X\qquad j=1,\ldots,n
$$
and points $z_j,w_j\in\mathbb{B}^1(K)$ such that $f(z_1)=x$, $f_n(w_n)=y$ and $f_j(w_j)=f_{j+1}(z_{j+1})$.
\item The \emph{Cherry-Kobayashi semi-distance} on $X(K)$ is defined by
$$d_{CK}(x,y)=\text{inf}\sum_{j=1}^n|z_j-w_j|\in\RR_{\geq0}\cup\{+\infty\}, \quad x,y \in X(K),$$
where the infimum is taken over all Kobayashi chains  joining $x$ and $y$.
\item We will say that the Cherry-Kobayashi semi-distance $d$ on $X(K$) is a \emph{distance} if, for all distinct $x$ and $y$ in $X(K)$, we have that $d(x,y) \neq 0$. (Even if $d$ is a distance,    $d(x,y)$ can nonetheless be infinite \cite[\S 2]{CherryKoba}.)
\end{enumerate}
\end{definition}

We start by showing that Cherry's semi-distance is a distance on an affinoid rigid analytic variety. 
\begin{proposition}\label{prop:aff_are_koba}
If $X$ is an affinoid rigid analytic variety over $K$, then $d$ is a distance  on $X(K)$.   
\end{proposition}

\begin{proof}
Let $X$ be $\Spa(R,R^\circ)$ with $R$ a Tate algebra $R=K\langle T_1,\ldots,T_n\rangle/I$. We can embed $X$ as a Zariski closed subvariety of $\mathbb{B}^n$. For any pair of points $x,y\in X(K)$ the set of Kobayashi chains joining them in $X$ is smaller than the one in $\mathbb{B}^n$. It then suffices to show that the semi-distance on $\mathbb{B}^n$ is  a distance, which is done in \cite[Example 2.9]{CherryKoba}.
\end{proof}
 
We now prove a    generalization of \cite[Lemma 2.13]{CherryKoba}.
\begin{proposition}\label{prop:CK}
 Let $X$ be a proper finitely presented scheme   over $\OO_K$. If the special fiber $X_k$ is pure, then the semi-distance on $X_K^{\an}(K)$ is a distance.
\end{proposition}
\begin{proof}
As we showed in Section \ref{section:inherit}, by the purity of $X_k$, any analytic map $\mathbb{B}^1\ra X_K^{\ad}$ induces a constant map when composed with the specialization morphism $X^{\ad}_K\ra X_k$. In particular, any Kobayashi chain is constant on the special fiber, and therefore factors over some open affinoid subvariety of $X_K^{\ad}$. The result then follows from the fact that Cherry's semi-distance on an affinoid is a distance  (Proposition \ref{prop:aff_are_koba}). 
\end{proof}

We now show that Cherry's semi-distance does not ``see'' the hyperbolicity of a space in general.  That is, roughly speaking, the next theorem says that  the  Cherry-Kobayashi semi-distance on a constant proper variety over $k(\!(t)\!)$ is a distance   if and only if the variety is pure. Therefore, this semi-distance is not enough to detect the hyperbolicity of a variety, as it only detects rational curves.

\begin{theorem}\label{thm:distance}	Let $k$ be an algebraically closed field of characteristic zero and let $X$ be a proper scheme over $k$. Fix a complete  non-archimedean valued field $K$ with residue field $k$ endowed with a section $k\hookrightarrow\OO_K$ of the quotient map.  The following are equivalent.
	\begin{enumerate}
		\item  $X$ is pure over $k$.
		\item $ X_K^{\an}$ is $K$-analytically pure.
		\item   For every dense open $C\subset \mathbb{P}^1_k$, every morphism of $K$-analytic spaces $C^{\an}\to X_K^{\an}$ is constant.
		\item The semi-distance on $X_K^{\an}(K)$ is a distance.
	\end{enumerate}
\end{theorem}
\begin{proof}  
	The equivalence of $(1)$ and $(2)$ follows from Corollary \ref{cor:char}.  The implication $(1)\implies (3)$ follows from Corollary \ref{thm:no_curves} (with $g=0$), and the implication $(3)\implies (2)$ is clear.
	  Proposition \ref{prop:CK} proves that $(1)$ implies $(4)$.   
	   Conversely, suppose there exists a non-constant map $\mathbb{P}^1\ra X$. Note that this morphism induces a  non-constant map $\mathbb{A}^{1\ad}_K\ra X_K^{\ad}$. Therefore,  the semi-distance on $X^{\an}_K(K)$ is not a distance by 
	  \cite[Corollary~2.7]{CherryKoba}.  This concludes the proof. 
\end{proof}

\section{Generizing and specialization of grouplessness}\label{sec:spec}

	Let $k$ be an algebraically closed field of characteristic zero. 
	
	\begin{remark}\label{remark:gt} Let us briefly say that a projective variety $X$ over $k$ is \emph{Kodaira-hyperbolic} if every integral subvariety of $X$ is of general type.  
  It is explained in \cite{vBJK} that Kodaira-hyperbolicity ``generizes'' and ``specializes''.  To be more precise, let $S$ be a smooth integral curve over $k$, and let $X\to S$ be a projective family   of varieties over $k$.  If there is an $s$ in $S(k)$ such that $X_s$ is Kodaira-hyperbolic, then the   generic fiber $X_{{K(S)}}$ is Kodaira-hyperbolic over   the function field $K(S)$ of $S$.
Furthermore, if the   generic fiber of $X\to S$ is Kodaira-hyperbolic and $k$ is uncountable, then there is a point $s$ in $S(k)$ such that $X_s$ is Kodaira-hyperbolic.
\end{remark}

Lang conjectured    that  a projective variety over $k$ is groupless if and only if $X$ is Kodaira-hyperbolic; see \cite{JBook, Lang2}. In particular, it predicts that the notion of  grouplessness  generizes and specializes, as   ``being Kodaira-hyperbolic'' generizes and specializes (Remark \ref{remark:gt}). In this section, we prove  these two predictions;  see Theorem    \ref{thm:gr_gen} and Theorem  \ref{thm:groupless_spec}.

\subsection{Purity and grouplessness generize}
We start by proving the generization property. We stress that  our proofs rely on non-archimedean analytic methods, even though the   statements are ``algebraic''.

\begin{theorem}[Purity generizes]\label{thm:pur_gen}  Let $k$ be an algebraically closed field and let $S$ be an integral noetherian normal scheme of characteristic zero. 
	Let $X\to S$ be a proper morphism. If there is a point $s$ in $S(k)$ such that $X_s$ is pure over $k$, then $X_{{K(S)}}$ is    pure.  
\end{theorem}
\begin{proof}
We can first localize at the point $s$ and assume that $S$ is local. By standard cutting arguments (see the proof of Corollary \ref{cor:purity_gens}) we can replace $S$ with the spectrum of a valuation ring  having residue field equal to $k$ and a complete, algebraically closed field of fraction $K$. We need to prove that if $X_k$ is pure, then $X_K$ is also pure. Since $k$ is of characteristic zero, the latter statement follows then from Proposition  \ref{prop:pur_gen_baby}.
\end{proof}

\begin{theorem}[Grouplessness generizes]\label{thm:gr_gen}    Let $k$ be an algebraically closed field and let $S$ be an integral noetherian normal scheme of characteristic zero.
	Let $X\to S$ be a proper morphism of  schemes. If there is a point $s$ in $S(k)$ such that $X_s$ is groupless over $k$, then $X_{K(S)}$ is    groupless over $K(S)$.  
\end{theorem}
\begin{proof} As in the proof of Theorem \ref{thm:pur_gen}, by localizing at the point $s$ and by standard cutting arguments (see the proof of Corollary \ref{cor:purity_gens}) we can replace $S$ with the spectrum of a valuation ring   having residue field equal to $k$ and a complete, algebraically closed field of fraction $K$.  

	By assumption, the special fiber $X_k$ of $X\to S$ is groupless. Therefore, by Theorem \ref{thm:main_result}, the proper scheme $X_K$ is $K$-analytically Brody hyperbolic. It follows readily that $X_K$ is groupless over $K$.  
\end{proof}
	 
	 \begin{proof}[Proof of Theorem \ref{thm:grouplessness_generizes}]
	 This  follows from (the more general) Theorem \ref{thm:gr_gen}. 
	 \end{proof}

 \subsection{Purity and grouplessness specialize}
 
In contrast to our proofs for the fact that purity and grouplessness ``generize'', we will  use  only algebraic  techniques to verify that purity and grouplessness   ``specialize''.

 \begin{theorem}[Purity specializes]\label{thm:purity_spec} Let $k$ be an uncountable algebraically closed field. Let $S$ be an integral variety over $k$. Let $X\to S$ be a projective morphism. Suppose that the   generic fiber of $X\to S$ is   pure. Then there is an $s$ in  $S(k)$ such that   $X_s$ is pure over $k$. 
 \end{theorem}
 \begin{proof}   Suppose that, for any $s$ in $S(k)$, the projective variety $X_s$ is not pure over $k$. To prove the theorem, it suffices to show that the geometric generic fiber of $X\to S$ is not pure.
 
 To do so,	consider the scheme $H:=\underline{\Hom}^{nc}_S(\mathbb{P}^1_S, X)$ parametrizing non-constant morphisms from $\mathbb{P}^1_S$ to $X$; see \cite[Section 4.c, pp. 221-19 -- 221-20]{Groth-FGA}.  Note that  $H$ is a countable union $H= \sqcup_{d\in \ZZ} H^d$ of finitely presented schemes $H^d$ over $S$.   Let $S_d$ be the image of $H^d(k)$ in $S(k)\subset S$. 
 	  Now, as $X_s$ is not pure over $k$ for any $s$ in $S(k)$, we have that $S(k) = \cup_{d\in \ZZ} S_d$.  The latter implies that  $$S = \overline{S(k)} =  \bigcup_{d\in \ZZ} \overline{S_d}.$$ 
 	   Since $k$ is uncountable and the right hand side is a countable union of closed subsets of $S$, 
 	  there is an integer $e$ such that $S =  \overline{S_e}$.  
 	  Thus, $S_e$ is dense in $S$ so that the morphism $H^e\to S$ is dominant. Therefore, the generic fiber of $H^e\to S$ is non-empty. This implies that the generic fiber of $\underline{\Hom}^{nc}_S(\mathbb{P}^1_S,X)\to S$ is non-empty, i.e., there is a non-constant morphism $\mathbb{P}^1_{{K(S)^a}}\to X_{{K(S)^a}}$ where $K(S)^a$ is an algebraic closure of $K(S)$.  
 	 Thus, the geometric generic fiber of $X\to S$ is not   pure. This concludes the proof.
 \end{proof}
 
 \begin{corollary}\label{cor:spec_pur}
 	Let $k$ be an uncountable algebraically closed field. Let $S$ be an integral variety over $k$. Let $X\to S$ be a projective morphism. Suppose that the generic fiber of $X\to S$ is   pure. Then the set of $s$ in $S(k)$ with $X_s$   pure is dense in $S$.  
 \end{corollary}
 \begin{proof}
 	Let $S'$ be the set of $s$ in $S(k)$ such that $X_s$ is pure over $k$. Suppose that $S'$ is not dense. Let $S^\circ$ be the complement of the closure of $S'$ in $S$. Then $S^\circ$ is a dense open of $S$. Let $X^\circ\to S^\circ$ be the restriction of $X\to S$ to $S^\circ$. Then the generic fiber of $X^\circ\to S^\circ$ is   pure (as it equals the generic fiber of $X\to S$). Thus, by Theorem \ref{thm:purity_spec}, there is an $s$ in $S^\circ(k)$ such that $X_s$ is pure over $k$. But then $s$ is an element of $S'$ leading to a contradiction.
 \end{proof}

 We are now ready to prove that grouplessness also specializes in families.  That is, let $k$ be an uncountable algebraically closed field and  let $S$ be an integral variety over $k$ with function field $K = K(S)$.  
 	Let $X\to S$ be a  projective morphism of schemes. We now show that, if $X_{K}$ is   groupless over $K$, then there is an $s$ in $S(k)$ such that $X_s$ is groupless.
 
 \begin{proof}[Proof of Theorem \ref{thm:groupless_spec}]   Let $A_g$ be the stack of principally polarized $g$-dimensional abelian varieties over $\ZZ$, and let $U_g\to A_g$ be the universal family. Note that the Hom-stack $\underline{\Hom}_{S\times A_g}(U_g, X\times A_g )\to S\times A_g$ is a countable disjoint union of finitely presented algebraic stacks. More precisely,  for every polynomial $P$ in $\QQ[t]$,
 	let $H_{g,P} =\underline{\Hom}^{P}_{S\times A_g}(U_g, X\times A_g)$ be the substack whose objects are  morphisms from $U_g\times S$ to $X\times A_g$ over $S\times A_g$ with Hilbert polynomial $P$. Note that $H_{g,P}\to A_g\times S$ is a finitely presented morphism of stacks \cite{OlssonHom}.  Let $S_{g,P}$ be the image of $H_{g,P}(k)$ in $S(k)$ via $H_{g,P}\to A_g\times S\to S$. (This map   associates to a $k$-point $s$ in $S(k)$, a principally polarized abelian variety $A$ over $k$ and a non-constant morphism $A\to X_s$ the point $s$ in $S(k)$.) Suppose that, for all $s$ in $S(k)$, the variety $X_s$ is not groupless. Then, by Lemma \ref{lem:zarhin}, for all $s$ in $S(k)$, there is a principally polarizable abelian variety $A$ over $k$ and a non-constant morphisms $A\to X_s$. Therefore, we have that
 	$$ S(k) = \bigcup_{(g,P)\in \ZZ_{\geq 0} \times \QQ[t]} S_{g,P}. $$ In particular, 
 	$$S = \overline{S(k)} =  \overline{\cup_{(g,P)\in \ZZ_{\geq 0} \times \QQ[t]}  S_{g,d}} = \bigcup_{(g,P)\in \ZZ_{\geq 0} \times \QQ[t]} \overline{S_{g,P}}.  $$
 	Since $k$ is uncountable and the right hand side is a countable union of closed subsets of $S$, there is an integer $g$ and a polynomial $P\in \QQ[t]$ such that   $\overline{S_{g,P}} = S$. This means that the morphism of stacks $H_{g,P}\to S$ is dominant. In other words, its generic fibre is non-empty. This means that $X_{{K(S)^a}}$ admits a non-constant morphism from some $g$-dimensional abelian variety, where  $K(S)^a$ is an algebraic closure of $K(S)$.  This proves the theorem.
 \end{proof}
 
 \begin{corollary} Let $k$ be an uncountable algebraically closed field. Let $S$ be an integral variety over $k$ with function field $K = K(S)$.
 	Let $X\to S$ be proper. If $X_{\overline{K}}$ is groupless, then the set of $s$ in $S(k)$ such that $X_s$ is groupless is dense in $S$. 
 \end{corollary} 
 \begin{proof}
 	This follows from the previous theorem (cf. the proof of Corollary \ref{cor:spec_pur}).
 \end{proof}
  
   We note that the results of this section imply that, for $S$ an integral noetherian scheme over $\mathbb{Q}$  and $X\to S$ a projective morphism,  the set of $s$  in $S$ such that $X_s$ is Zariski-countable open, as defined in \cite{vBJK}.

 \section{A Theorem of the Fixed Part in mixed characteristic}
 In this section we let $K$ be a complete algebraically closed non-archimedean valued field of characteristic zero whose residue field is of characteristic $p>0$.
 
Recall that any complex analytic variety which is uniformized by some  bounded  domain satisfies (a consequence of) the Theorem of the Fixed Part; see Theorem \ref{thm:fixed_part_intro} in the introduction.  In this section we prove a non-archimedean analogue of this statement (Theorem \ref{thm:moduli_of_abelian_varieties_1}).  
 To do so, we start with a brief discussion of inverse limits and fundamental groups in the category of adic spaces. 
 
   The category of adic spaces does not have arbitrary inverse limits (cf. \cite[Definition~2.4.1]{ScholzeWeinstein}). An adequate replacement of this notion in certain cases (without appealing to diamonds \cite{ScholzeDiamonds}) is introduced in \cite[Definition 2.4.2]{huber}, and we recall it here briefly.
   \begin{definition}
let $\{X_m\}_{m\in I}$ be a cofiltered inverse
system of adic spaces. We say that an adic space $X$ endowed with compatible maps $X\ra X_m$ is \emph{similar to the projective limit }and we write $X\sim \varprojlim_XX_m$ if on the underlying topological spaces we have  $|X|\cong\varprojlim|X_m|$  and if there is an open cover of $X$ by affinoid subsets $\Spa(A, A^+)$ such that the map of rings $\varinjlim A_i\ra A$ has a dense image, where the limit runs over the affinoid open subsets $\Spa(A_i,A_i^+)$ of some $X_m$ over which the map $\Spa(A,A^+)\ra X_m$ factors.
   \end{definition}

     We will state our theorem using \'etale fundamental groups of analytic varieties. We recall that
   if $X$ is a  connected noetherian scheme, then $\pi_1^{\mathrm{et}}(X)$ denotes the \'etale fundamental group of $X$ (with respect to the choice of some geometric base point of $X$); see \cite{SGA1}. 
Analogously,   
  if $X$ is a connected rigid analytic variety over $K$,  we denote by $\pi^{\alg}_1(X,x)$  the \emph{algebraic \'etale fundamental group} of $X$, i.e., the pro-finite group attached to the category of finite \'etale covers of $X$ with respect to some chosen geometric point $x$ (see \cite[Theorem~2.9]{deJo-fg} and \cite[Section III.1.4.1]{andre-per}). Note that, by   \cite[Proposition III.1.4.4]{andre-per}, the isomorphism class of the algebraic \'etale fundamental group of $X$ is independent of the choice of base-point. We will therefore omit the base-point from our notation. 
     If $X$ is a connected finite type scheme over  $K$, then the analytification functor induces an isomorphism between $\pi^{\alg}_1(X^{\ad})$  and the \'etale fundamental group $\pi_1^{\mathrm{et}}(X)$  of $X$ (see \cite[Theorem 3.1]{luk}). 
  
  \begin{example}   
If $K$ is a non-archimedean complete algebraically closed field  of characteristic zero, then $$\pi_1^{\mathrm{et}}(\mathbb{G}_{m,K}) = \pi_1^{\alg}(\mathbb{G}_{m,K}) = \widehat{\mathbb{Z}}.$$
 \end{example}

   Using the strategy sketched in the introduction, inspired by the complex case, we now prove  the following result.

   \begin{proposition}\label{prop:thm_of_fixed_part} Let $S\to X_0$ be a morphism of connected reduced rigid analytic varieties over $K$.
Let $\{X_m\}_{m = 0}^\infty$ be a cofiltered inverse
system of connected rigid analytic varieties   over $K$ with finite \'etale transition maps.   Assume that there is a perfectoid 
space $P$ over $K$ such that $P\sim \varprojlim_m X_m$.  
Let $\Gamma_m \subset \pi_1^{\alg}(X_0)$ be the subgroup associated to $X_m\to X_0$, and suppose that  the  image of the induced morphism
on (algebraic) \'etale fundamental groups 
\[
\pi_1^{\alg}(S) \to \pi_1^{\alg}(X_0)
\]
lies in $\bigcap_{m=0}^\infty \Gamma_m$. Then $S\to X_0$ is constant.
   \end{proposition}

   \begin{proof}
By our assumption on the image of $\pi_1^{\alg}(S)\to \pi_1^{\alg}(X_0)$, the morphism lifts to every  finite \'etale cover $X_m$ of $X_0$. Since $S$ is reduced, hence stably uniform \cite[Section 3]{buzz-ver}, we   conclude by \cite[Proposition 2.2]{ludwig} that the  morphism lifts to a map $S\ra P$. The latter is  constant by Proposition \ref{prop:rigid_to_perf}, so that $S\to X_0$ is constant, as required.
   \end{proof}
   
  We follow the notation of Section \ref{section:app_to_mod}, and let $\mathcal{A}_g$ be the stack of $g$-dimensional principally polarized abelian schemes over $\ZZ$. Let $N>3$ be coprime to the residue characteristic of $K$.
  Let $X:=\mathcal{A}_{g,K}^{[N]}$, and note that $X$ is a smooth quasi-projective scheme over $K$.

Consider the principal congruence subgroups
\begin{eqnarray*} 
	\Gamma(p^n) &=& \left\{ \gamma \in \mathrm{GSp}_{2g}(\mathbb Z_p) \ | \ \gamma \equiv \left( \begin{array}{cc} 1 &0 \\ 0 & 1 \end{array} \right) \mod p^n  \right\}  
\end{eqnarray*} If $n\geq 1$ is an integer,    we let $X_{\Gamma(p^n)}$ be the corresponding moduli space of principally polarized $g$-dimensional abelian varieties with full level $N$-structure and with level $\Gamma(p^n)$-structure over $K$.

\begin{theorem}[Scholze]\label{thm:scholze}
There is a unique perfectoid space $X_{\Gamma(p^\infty)}$ such that   
	\[X_{\Gamma(p^\infty)} \sim \varprojlim_n X^{\ad}_{\Gamma(p^n)}.\] 
\end{theorem}
\begin{proof}
This   follows     from \cite[Theorem~III.1.2]{ScholzeTorsion}. 
 \end{proof}
 
 As we show now,   the Theorem of the Fixed Part for the moduli space of principally polarized abelian varieties (Theorem \ref{thm:moduli_of_abelian_varieties_1}) is   a consequence of Scholze's theorem 
  and Proposition \ref{prop:thm_of_fixed_part}.
  
  \begin{proof}[Proof of   Theorem \ref{thm:moduli_of_abelian_varieties_1}]
  Combine Proposition \ref{prop:thm_of_fixed_part} and  Theorem \ref{thm:scholze}.
  \end{proof}
   
   We conclude this section with other applications of Proposition \ref{prop:thm_of_fixed_part}.
  
  \begin{example} Let $\mathbb{T}^n$ be $\Spa K\langle T_1^{\pm 1},\ldots, T_n^{\pm 1}\rangle$, and let $X$ be \'etale over $\mathbb{T}^n$. Define $X_0 := X$. For $m>0$, let $f_m:\mathbb{T}^n\to \mathbb{T}^n$ be the $m$th power map, and define $X_m:= X\times_{\mathbb{T}^n, f_m} \mathbb{T}^n$. Then the inverse limit of the cofiltered inverse system $\{X_m\}_{m=0}^\infty$  is a perfectoid space \cite[Lemma 4.5]{scholze-ph}. We conclude that $X$ is hyperbolic (Proposition \ref{prop:aff_are_brody}), and satisfies the Theorem of the Fixed Part (Proposition \ref{prop:thm_of_fixed_part}).
  \end{example}
  
 \begin{example} \label{ex:gm}
 Consider the inverse system of rigid analytic varieties $\{X_m\}_{m=0}^\infty$, where $X_m:=\mathbb{G}_{m,K}^{\an}$ and the (finite \'etale) morphism $X_m\to X_{m-1}$ is given by $z\mapsto z^p$.
 There is a perfectoid space $\mathbb{G}_{m}^{\mathrm{perf}}$   such that $\mathbb{G}_m^{\mathrm{perf}}\sim  \varprojlim_m X_m$ (see \cite[Claim~1.4]{ScholzePerfectoid}); we refer to $\mathbb{G}_m^{\mathrm{perf}}$  as the perfection of $\mathbb{G}_{m,K}$.
   Since $\mathbb{G}_m^{\mathrm{perf}}$ is a perfectoid space, it is $K$-analytically Brody hyperbolic (Proposition \ref{prop:rigid_to_perf}). Moreover, by Proposition \ref{prop:thm_of_fixed_part}, the curve $\mathbb{G}_{m,K}$ satisfies the Theorem of the Fixed Part. 
   
   Now, the above shows that $\mathbb{G}_{m,K}^{\ad}$ is $K$-analytically Brody hyperbolic \textit{up to} a pro-finite \'etale cover. However, it is clear that $\mathbb{G}_{m,K}$ is not  $K$-analytically Brody hyperbolic. This shows that an adic space which is $K$-analytically Brody hyperbolic up to a pro-finite \'etale cover is not necessarily  $K$-analytically Brody hyperbolic. (Thus, the analogue of Proposition \ref{prop:chevalley_weil} for pro-finite \'etale morphisms fails.)
 \end{example}

 \begin{example}  We give  an example similar to Example \ref{ex:gm}.
Let $A$ be an abelian variety over $K$. Define the inverse system of rigid analytic varieties $\{X_m\}_{m=0}^\infty$ by $X_m := A^{\an}$ and the (finite \'etale) morphism $X_m\to X_{m-1}$ to be multiplication by $p$. Then, there is a perfectoid space $A^{\mathrm{perf}}$ such that $A^{\mathrm{perf}}\sim  \varprojlim_m X_m$; see  \cite[Theorem~1]{AbVarPerfectoid}. In particular, the abelian variety $A$ satisfies the Theorem of the Fixed Part (Proposition \ref{prop:thm_of_fixed_part}), and is \emph{up to} a pro-finite \'etale morphism a $K$-analytically Brody hyperbolic variety. However, $A$ is (clearly) not $K$-analytically Brody hyperbolic.
 \end{example}
 
A complex analytic space which has a topological covering by a Brody hyperbolic space is itself Brody hyperbolic. However,   the last two examples above show that an adic space over $K$ which has a pro-finite \'etale cover by a $K$-analytically Brody hyperbolic space (e.g., a perfectoid space)    is itself not necessarily  $K$-analytically Brody hyperbolic. Therefore, the last two examples above show that  ``descending'' the $K$-analytic Brody hyperbolicity of $X_{\Gamma(p^\infty)}$ to the moduli space $X$  is a non-trivial endeavour. More precisely, to show that $X$ is  $K$-analytically Brody hyperbolic, one has to   show that the ``monodromy'' of every morphism $\mathbb{G}_m^{\an}\to X^{\an}$ is trivial, and the latter does not follow from the existence of a pro-finite \'etale cover by a perfectoid space.

\bibliography{refs}{}
\bibliographystyle{plain}
\end{document}